\documentclass[11pt,oneside,dvipsnames]{amsart}
\usepackage{stackrel}
\usepackage[french, english]{babel}
\usepackage{a4wide}
\usepackage[utf8,latin1]{inputenc}
\usepackage{times}
\usepackage{adjustbox}
\usepackage{float}
\usepackage[cyr]{aeguill}
\usepackage{amsmath,amscd}
\usepackage{amssymb}
\usepackage{mathrsfs}
\usepackage{bbold}
\usepackage{amsthm}
\usepackage{geometry}
\usepackage{graphicx}
\usepackage{epstopdf}
\usepackage{amsfonts}
\usepackage{amsmath}
\usepackage{amsmath,mathtools}
\usepackage{amsmath,graphicx}
\usepackage{amscd}
\usepackage{dsfont}
\usepackage{latexsym}
\usepackage{verbatim}
\usepackage{comment}
\usepackage{tikz-cd}
\usetikzlibrary{matrix,arrows}
\usepackage{tikz}
\usetikzlibrary{tikzmark}
\usepackage{tensor}
\usepackage{leftidx}
\usepackage[retainorgcmds]{IEEEtrantools}
\usepackage[title]{appendix}
\usepackage[mathscr]{euscript}
\usepackage{bbm}
\usepackage{url}
\usepackage{upgreek}
\usepackage{mathtools, nccmath}
\usepackage[utf8]{inputenc}
\usepackage{multirow}
\usepackage{rotating}
\usepackage{marginnote}
\usepackage{hyperref} 

\usepackage[all]{xy}
\usepackage[T1]{fontenc}

\usepackage{hyperref}
\usepackage{array}
\usepackage{tabularx}
\usepackage{yfonts}
\usepackage{setspace}  
\usepackage{geometry}
\makeatletter
\renewcommand\subsubsection{\@startsection{subsubsection}{3}{\z@}%
                                     {-3.25ex\@plus -1ex \@minus -.2ex}%
                                     {-0.5em}
                                     {\normalfont\normalsize\bfseries}}
\makeatother

\theoremstyle{plain}
\newtheorem{Thm}{Theorem}
\newtheorem{Prop}[Thm]{Proposition}
\newtheorem{Lem}[Thm]{Lemma}

\theoremstyle{definition}

\newtheorem{Rem}[Thm]{Remark}

\newtheorem{Defn}[Thm]{Definition}
\newtheorem{Eg}[Thm]{Example}

\newtheoremstyle{named}{}{}{}{}{\bfseries}{.}{.5em}{\thmnote{#3}#1}
\theoremstyle{named}

\DeclareMathOperator{\ad}{ad}

\DeclareMathOperator{\Aut}{Aut}

\DeclareMathOperator{\bcdot}{\!\cdot\!}

\DeclareMathOperator{\Ch}{Ch}

\DeclareMathOperator{\diag}{diag}

\DeclareMathOperator{\gl}{\mathfrak{gl}}
\DeclareMathOperator{\GL}{GL}

\DeclareMathOperator{\Hom}{Hom}

\DeclareMathOperator{\Id}{Id}
\DeclareMathOperator{\Ima}{Im}
\DeclareMathOperator{\Inn}{Inn}

\DeclareMathOperator{\isom}{\!\!\smash{\begin{array}{c}\sim\\[-1em]
\rightarrow\end{array}}\!\!}
\DeclareMathOperator{\lisom}{\!\!\smash{\begin{array}{c}\sim\\[-1em]
\longrightarrow\end{array}}\!\!}

\DeclareMathOperator{\kk}{\mathds{k}}

\DeclareMathOperator{\Out}{Out}
\DeclareMathOperator{\PGL}{PGL}

\DeclareMathOperator{\Rep}{Rep}

\DeclareMathOperator{\SL}{SL}

\DeclareMathOperator{\Stab}{Stab}

\DeclareMathOperator{\tr}{Tr}

\DeclareMathOperator{\ud}{\mathrm{d}\!}
\DeclareMathOperator{\lb}{\!<\!}
\DeclareMathOperator{\rb}{\!>\!}
\DeclareMathOperator{\ds}{\!/\mkern-2mu/\mkern-2mu}

\newcommand{\mc}{\multicolumn{1}{c}}

\makeatletter
\newcommand*{\rom}[1]{\expandafter\@slowromancap\romannumeral #1@}
\makeatother

\makeatletter
\let\orgdescriptionlabel\descriptionlabel
\renewcommand*{\descriptionlabel}[1]{%
  \let\orglabel\label
  \let\label\@gobble
  \phantomsection
  \edef\@currentlabel{#1}%
  \let\label\orglabel
  \orgdescriptionlabel{#1}%
}
\makeatother

\usepackage{subfig}	 

\title{On Character Varieties with Non-Connected Structure Groups}
\author{Cheng Shu}

\usepackage{pxfonts}
\usepackage{indentfirst}
\usepackage{hyperref}
\colorlet{ivory}{Apricot!30!}
\colorlet{space}{black!85!}
\definecolor{bgc}{RGB}{29, 44, 46}
\definecolor{txt}{RGB}{223, 222, 189}
\definecolor{cmd}{RGB}{206, 151, 88}
\begin{document}
\let\bs\boldsymbol
\maketitle
\begin{abstract}
For any connected complex reductive group $G$, any finitely generated discrete group $\Pi$ and a normal subgroup $\tilde{\Pi}$ with quotient group $\Gamma$, we study the associated $G\rtimes\Gamma$-character variety, the space of admissible $G\rtimes\Gamma$-representations of $\Pi$. We study the relation between this variety and the $\Gamma$-fixed points in the usual $G$-character variety associated to $\tilde{\Pi}$. In the process, we give the classification of the isomorphism classes of the semi-direct products $G\rtimes\Gamma$, with fixed $G$ and $\Gamma$. In the case where $\tilde{\Pi}$ and $\Pi$ are fundamental groups of Riemann surfaces, a genericity condition on the conjugacy classes of monodromy is introduced as a sufficient condition for the irreducibility of $G\rtimes\Gamma$-representations. The example of $\GL_n\rtimes\lb\sigma\rb~$-character varieties is discussed in detail.
\end{abstract}
\setcounter{tocdepth}{1}
\tableofcontents
\addtocontents{toc}{\protect\setcounter{tocdepth}{-1}}
\section*{Introduction}
Given a connected complex reductive group $G$ and a Riemann surface $X$, the Riemann-Hilbert correspondence and the non abelian Hodge correspondence identify three moduli spaces: the moduli space of Higgs $G$-bundles on $X$, the moduli space of flat $G$-connections on $X$, and the $G$-character variety of $X$. (This picture generalises to the case of punctured Riemann surfaces by introducing restrictions on monodromies around the punctures.) We will call $G$ the structure group of the character variety. We may regard $G$-bundles as torsors under the constant group scheme $G\times X$ and generalise the situation by considering a non constant (i.e. non-split) group scheme on $X$. For example,  torsors under a unitary group scheme on $X$ as in \cite{LN}. In this more general setting, the corresponding character varieties should have a structure group of the form $G\rtimes \Gamma$, where $\Gamma$ is the Galois group of a finite Galois covering $\tilde{X}/X$ such that the non constant group scheme lifts to a constant one. (A complete classification of the isomorphism classes of the groups of the form $G\rtimes\Gamma$ with given $G$ and $\Gamma$ is given by Theorem \ref{para-GGamma}, which may be of independent interest. See \cite[Theorem 5.1]{BE} for a theorem of Taunt in the context of finite groups.) The definition of such a character variety is simple, as is explained below.

Fix a homomorphism $\psi:\Gamma\rightarrow\Aut G$ so that we have a semi-direct product $G\rtimes\Gamma$. The representation variety consists of those homomorphisms $\pi_1(X)\rightarrow G\rtimes\Gamma$ that make (the right hand side of) the following diagram commute, and these homomorphisms will be called \textit{admissible} $G\rtimes\Gamma$-representations.  
\begin{center}
\begin{tikzpicture}[node distance=2cm, auto]
  \node (A) {$1$};
  \node [right of=A] (B) {$\pi_1(\tilde{X})$};
  \node [right of=B] (C) {$\pi_1(X)$};
  \node [right of=C] (D) {$\Gamma$};
  \node [right of=D] (E) {$1$};
  
  \node [below of=A, node distance= 1.5cm] (A1) {$1$};
  \node [right of=A1] (B1) {$G$};
  \node [right of=B1] (C1) {$G\rtimes\Gamma$};
  \node [right of=C1] (D1) {$\Gamma$};
  \node [right of=D1] (E1) {$1$};
  
  \draw[->] (B) to node {$\tilde{\rho}$} (B1);
  \draw[->] (C) to node {$\rho$} (C1);
  \draw[->] (D) to node {$=$} (D1);
  \draw[->] (A) to node {} (B);
  \draw[->] (B) to node {} (C);
  \draw[->] (C) to node {} (D);
  \draw[->] (D) to node {} (E);
  \draw[->] (A1) to node {} (B1);
  \draw[->] (B1) to node {} (C1);
  \draw[->] (C1) to node {} (D1);
  \draw[->] (D1) to node {} (E1);
  
\end{tikzpicture}.
\end{center}
An admissible $\rho:\pi_1(X)\rightarrow G\rtimes\Gamma$ can be restricted to $\tilde{\rho}:\pi_1(\tilde{X})\rightarrow G$. We will call $\tilde{\rho}$ the \textit{underlying $G$-representation} of $\rho$. The conjugation of $G$ on $G\rtimes\Gamma$ induces an action on this variety and the corresponding categorical quotient is called the $G\rtimes\Gamma$-character variety. 

In this article, we study a couple of facets of this variety. As in the case of usual character varieties, the closed orbits and stable orbits in the representation variety consist exactly of \textit{semi-simple} representations and \textit{irreducible} representations, and the points in the character variety are in bijection with the closed orbits. The semi-simple and irreducible $G\rtimes\Gamma$-representations are defined in terms of completely reducible and irreducible subgroups of non-connected algebraic groups. The definition of $G\rtimes\Gamma$-character varieties as well as these fundamental results are given in Section \ref{GGAMMACV}. The classification of the semi-direct products $G\rtimes\Gamma$, with fixed $G$ and $\Gamma$, is given in \S \ref{GGAMMA}.
  
It is tempting to ask about the relation between an admissible $G\rtimes\Gamma$-representation $\rho$ and its underlying $G$-representation $\tilde{\rho}$. For simplicity, in the rest of this introduction we assume $G=\GL_n(\mathbb{C})$. Denote by $\Ch(\tilde{X},G)$ the usual $G$-character variety associated with $\tilde{X}$ and denote by $\Ch^{\circ}(\tilde{X},G)$ the open subvariety of irreducible $G$-representations. With the homomorphism $\psi:\Gamma\rightarrow\Aut G$ that defines the semi-direct product $G\rtimes\Gamma$, we have a $\Gamma$-action on $\Ch(\tilde{X},G)$. For the definition of this action, see (\ref{sigma-rho}) and Remark \ref{GammaAction}. We present below our first result, which concerns the relation between $G\rtimes\Gamma$-character varieties associated to $\tilde{X}/X$ and the variety $\Ch^{\circ}(\tilde{X},G)^{\Gamma}$ of $\Gamma$-fixed points. 

Denote by $\Omega(\psi)$ the composition of $\psi$ and the natural homomorphism $\Aut G\rightarrow\Out G$. The fixed-points locus $\Ch(\tilde{X},G)^{\Gamma}$ only depends on $\Omega(\psi)$. The equivalence classes of those $\psi'$ such that $\Omega(\psi')=\Omega(\psi)$ are parametrised by the pointed set $H^1(\Gamma,\Inn G)$ with base point $\psi$, where $\Inn G=G/Z_G$ and $Z_G$ is the centre of $G$. An element of $H^1(\Gamma,\Inn G)$ will be represented by some $\psi'$. Let $\delta:H^1(\Gamma,\Inn G)\rightarrow H^2(\Gamma,Z_G)$ be the natural map.  

\begin{Thm}\label{B}
For any semi-simple representation $\tilde{\rho}:\tilde{\Pi}\rightarrow G$, denote by $[\tilde{\rho}]$ its $G$-orbit, i.e. the corresponding element of $\Ch(\tilde{\Pi},G)$. We have:
\begin{itemize}
\item[(i)] For each $[\tilde{\rho}]\in\Ch^{\circ}(\tilde{\Pi},G)^{\Gamma}$, there is a well-defined element $\mathfrak{c}([\tilde{\rho}])\in H^2(\Gamma,Z_G)$. This defines a partition of $\Ch^{\circ}(\tilde{\Pi},G)^{\Gamma}$;
\item[(ii)] $\tilde{\rho}$ is the underlying $G$-representation of some admissible $G\rtimes\Gamma$-representation of $\Pi$ if and only if $\mathfrak{c}([\tilde{\rho}])^{-1}=\delta(\psi')$ for some $\psi'$ such that $\Omega(\psi')=\Omega(\psi)$. In this case, the semi-direct product $G\rtimes\Gamma$ is defined by $\psi'$.
\end{itemize}
\end{Thm}
\begin{Rem}\label{Question}
We do not know whether the map $[\tilde{\rho}]\mapsto\mathfrak{c}([\tilde{\rho}])$ is surjective. We do not know whether a $G\rtimes_{\psi}\Gamma$-character variety can be isomorphic to a $G\rtimes_{\psi'}\Gamma$-character variety if $\delta(\psi)\ne\delta(\psi')$.
\end{Rem}\noindent
For the proof, see Proposition \ref{defn-c(rho)} and Proposition \ref{changepsi}. (By Lemma \ref{imageofdelta} below, the isomorphism class of $G\rtimes_{\psi'}\Gamma$ only depends on $\delta(\psi')$.) The first part of the theorem already appeared in \cite{Wu} and \cite{Sch}.

Our most important examples of such character varieties are the $\GL_n\rtimes\lb\sigma\rb~$-character varieties associated to a double covering $\tilde{X}\rightarrow X$, where $\sigma$ is an order 2 outer automorphism of $\GL_n$. These are character varieties that are unitary in the global sense. Suppose that $\tilde{X}\rightarrow X$ is exactly the unbranched part of a branched double covering $\tilde{X}'\rightarrow X'$ of compact Riemann surfaces. If we restrict the monodromy around the punctures (removed ramification points) to some $\GL_n$-conjugacy classes in the connected component $\GL_n\sigma$ denoted by $(C_j)_j$, then we obtain the character variety $\Ch_{\Gamma,\mathcal{C}}(X,G)$ as defined in \S \ref{mathbfssigma}. It can be written as
\begin{equation}\label{intro-var1}
\{(A_i,B_i)_i(Y_j)_j\in \GL_n^{2g}\times \prod_j C_j|\prod_{i=1}^g[A_i,B_i]\prod_{j}Y_j=1\}\ds\GL_n,
\end{equation}
where $g$ is the genus of $X$. Note that there are necessarily an even number of ramification points in the branched covering $\tilde{X}'\rightarrow X'$ so that $\prod_jY_j$ indeed lies in the identity component $\GL_n$. We may also start with an unbranched covering $\tilde{X}'\rightarrow X'$ and define $\tilde{X}\rightarrow X$ by removing a finite set. In this case, we can restrict the monodromy around the punctures to some conjugacy classes of $\GL_n$, denoted by $(C_j)_j$. Now the character variety $\Ch_{\Gamma,\mathcal{C}}(X,G)$ can be written as
\begin{equation}\label{intro-var2}
\{(A_i,B_i)_i(X_j)_j\in \GL_n^{2g}\times \prod_{j=1} C_j\mid A_1\sigma(B_1)A_1^{-1}B_1^{-1}\prod_{i=2}^g[A_i,B_i]\prod_{j=1}X_j=1\}\ds\GL_n.
\end{equation}

Since these character varieties are also defined by a simple equation resembling the usual character variety, its E-polynomial can be calculated following the method in \cite{HLR}, which involves point-counting over finite fields. One of the major motivation of this article is to prepare some technical results for this point-counting project, especially the notion of \textit{generic conjugacy classes}. The generic condition is a sufficient condition for all elements of the character variety to be irreducible $\GL_n\rtimes\lb\sigma\rb~$-representations. Beware that this does not imply that the underlying $\GL_n(\mathbb{C})$-representations are irreducible, so working with non-connected groups is essential. Below we give an explicit description of generic tuple of conjugacy classes in the case of $\GL_n(\mathbb{C})\rtimes\lb\sigma\rb~$-character varieties. The definition for general reductive group $G$ and finite group $\Gamma$ will be given in Section \ref{GENCONJ}, which will be the most technical part of the article.

We first restrict ourselves to the case of branched covering and two semi-simple classes $(C_1,C_2)$ in $\GL_n\sigma$. According to the parametrisation of semi-simple conjugacy classes in the connected component $\GL_n(\mathbb{C})\sigma$, they have representatives of the form $$\diag(a_1,\ldots,a_N,a_N^{-1},\ldots,a_1^{-1})\sigma,\quad\diag(b_1,\ldots,b_N,b_N^{-1},\ldots,b_1^{-1})\sigma,$$ for some $a_i$, $b_i\in\mathbb{C}^{\ast}$, $1\le i\le N$. (So we have assumed $n=2N$.) Then the genericity condition says the following. For any integer $1\le l\le N$, any two subsets $\mathbf{A}$, $\mathbf{B}\subset\{1,2,\ldots,N\}$ of cardinality $l$, and any two $l$-tuples of signs $(e_i)_{i\in \mathbf{A}}$, $(e'_i)_{i\in \mathbf{B}}$, with $e_i$, $e'_i\in\{\pm 1\}$, we have 
\begin{equation}\label{intro-bg}
\prod_{i\in \mathbf{A}}a_i^{2e_i}\prod_{i\in \mathbf{B}}b_i^{2e'_i}\ne 1.
\end{equation} In the case of unbranched covering, suppose that we have two semi-simple classes $(C_1,C_2)$ in $\GL_n$ with eigenvalues $$(a_1,\ldots,a_n),\quad(b_1,\ldots,b_n).$$ Now the genericity condition says the following. For any $1\le l \le N$, any four subsets $\mathbf{A}_1,\mathbf{A}_2,\mathbf{B}_1,\mathbf{B}_2$ of $\{1,\ldots,n\}$ such that 
\begin{itemize}
\item
$|\mathbf{A}_1|=|\mathbf{A}_{2}|=|\mathbf{B}_1|=|\mathbf{B}_{2}|=l$;
\item
$\mathbf{A}_1\cap \mathbf{A}_2=\varnothing$, $\mathbf{B}_1\cap \mathbf{B}_2=\varnothing$,
\end{itemize}
 we have
\begin{equation}\label{intro-ug}
\prod_{i\in \mathbf{A}_1}a_i\prod_{i\in \mathbf{B}_1}b_i\prod_{i\in \mathbf{A}_2}a_i^{-1}\prod_{i\in \mathbf{B}_2}b_i^{-1}\ne 1.
\end{equation}
\begin{Rem}
In the unbranched case, we have a tuple of conjugacy classes of $\GL_n$, very much the same as in the case of usual $\GL_n$-character variety. However, the definitions of generic conjugacy classes in these two cases are very different. Recall that for the usual $\GL_n$-character variety, the generic condition for $(C_1,C_2)$ says: for any $1\le l < n$, any two subsets $\mathbf{A},\mathbf{B}$ of $\{1,\ldots,n\}$ such that $|\mathbf{A}|=|\mathbf{B}|=l$, we have 
\begin{equation}\label{intro-us}
\prod_{i\in \mathbf{A}}a_i\prod_{i\in \mathbf{B}}b_i\ne 1.
\end{equation}
The generic condition for one single conjugacy class is obtained by removing the $\mathbf{B}_1$ and $\mathbf{B}_2$ terms in (\ref{intro-ug}), and the $\mathbf{B}$ terms in (\ref{intro-us}).
As an example, let us consider the unbranched case where $n=2$, $g=1$, and we have a single conjugacy class $C=\{-\Id\}$. This class is generic for the usual $\GL_2$-character variety, but not for the $\GL_2\rtimes\lb\sigma\rb~$-character variety in the unbranched case. In fact,  in the $\GL_2\rtimes\lb\sigma\rb~$-character variety defined by
$$A\sigma(B)A^{-1}B^{-1}X=\Id,$$with $X=-\Id$, there exists an element defined by $A=\Id$, $B=\diag(\sqrt{-1},\sqrt{-1})$, which is obviously not an irreducible representation. For the precise definition of irreducible $\GL_n\rtimes\lb\sigma\rb~$-representation, see \S \ref{GGAMMACV}.
\end{Rem} 

In the last section, we apply our general results given in previous sections to $\GL_n\rtimes\lb\sigma\rb~$-character varieties. We will see that even in this simple case, the statements of Theorem \ref{B} are non-trivial.  We will also prove that the $\GL_n(\mathbb{C})\rtimes\lb\sigma\rb~$-character varieties with generic conjugacy classes are smooth, and give a dimension formula. Let $\Ch_{\Gamma,\mathcal{C}}(X,\GL_n)$ denote the variety defined by (\ref{intro-var1}) or (\ref{intro-var2}).
\begin{Thm}
If the tuple of conjugacy classes $\mathcal{C}=(C_j)_j$ is generic, then the variety $\Ch_{\Gamma,\mathcal{C}}(X,\GL_n)$ is smooth and its dimension is given by $$
\dim\Ch_{\Gamma,\mathcal{C}}(X,\GL_n)=(2g-2)\dim\GL_n+\sum_{\text{all }j}\dim C_j.$$
\end{Thm}

Finally, let us remark that Boalch and Yamakawa have studied the moduli space of twisted Stokes representations in \cite{BY}. In  the "tame" case, i.e. trivial Stokes data case, the Stokes representations are reduced to torsors under a non-constant local system of groups on $X$. So our character variety is the tame case of theirs.

\subsection*{Acknowledgement}
This article is part of my thesis prepared at Universit\'e de Paris. I thank Florent Schaffhauser for many helpful discussions which initiated this work. Some ideas in this article resembles \cite[\S 4]{Sch2}. I thank Fran\c cois Digne for reading an earlier version of the article and pointing out a mistake. I would like to thank Philip Boalch and Jean Michel for answering some questions, and my thesis advisor Emmanuel Letellier for pointing out a mistake. I thank the anonymous referees for many suggestions and comments that significantly improved the organisation of the article, and pointing out many typos.

\numberwithin{Thm}{section}
\numberwithin{equation}{subsection}
\addtocontents{toc}{\protect\setcounter{tocdepth}{1}}
\section{Preliminaries}
Throughout the article, the base field will be $\kk=\mathbb{C}$ unless otherwise stated. We begin by collecting some results on non-connected algebraic groups. 
\subsection{}
Let $G$ be a linear algebraic group, which is not necessarily connected. The identity component of $G$ will be denoted by $G^{\circ}$. An algebraic group $G$ is called reductive if $G^{\circ}$ is reductive. For any subgroup $H\subset G$ and subsets $X$, $Y\subset G$, we will write $C_H(X)=\{g\in H\mid gx=xg\text{ for all $x\in X$}\}$, $N_H(X)=\{g\in H\mid gX=Xg\}$ and $N_H(X,Y)=\{g\in H\mid gX=Xg,~gY=Yg\}$. If $G$ is connected and $T\subset G$ is a maximal torus, then $W_G(T)$ denotes its Weyl group defined by $T$.

A closed subgroup $P\subset G$ is parabolic if $G/P$ is complete. By \cite{Spr} Lemma 6.2.4, $P$ is parabolic in $G$ if and only if $P^{\circ}$ is parabolic in $G^{\circ}$. Any $G$ is the semi-direct product of its unipotent radical and a Levi factor. A closed subgroup of $G$ is called a Levi subgroup if it is a Levi factor of some parabolic subgroup. Given a parabolic subgroup $P^{\circ}\subset G^{\circ}$, $N_G(P^{\circ})$ is the largest parabolic subgroup of $G$ that has $P^{\circ}$ as its identity component. Note that $P^{\circ}$ itself is also a parabolic subgroup of $G$. 

\subsection{}\label{PmeetsG1}
In general, for an arbitrary parabolic subgroup $P^{\circ}$ of $G^{\circ}$, $P:=N_G(P^{\circ})$ does not necessarily meet all connected components of $G$. The connected components of $G$ that meet $P$ are determined as follows.

Let $G^1$ be a connected component of $G$ and denote by $\mathcal{P}$ the $G^{\circ}$-conjugacy class of $P^{\circ}$. Observe that the conjugation of $G^1$ on $G^{\circ}$ induces a well-defined bijection from the set of $G^{\circ}$-conjugacy classes of parabolic subgroups of $G^{\circ}$ to itself. Then $P$ meets $G^1$ if and only if some element of $G^1$ normalises $P^{\circ}$, if and only if $G^1$ leaves $\mathcal{P}$ stable. The set of conjugacy classes of parabolic subgroups of $G^{\circ}$ are in bijection with the set of subsets of vertices of the Dynkin diagram of $G^{\circ}$ (if $G^{\circ}$ is reductive). Therefore, $P$ meets $G^1$ if and only if the Dynkin subdiagram corresponding to $\mathcal{P}$ is stable under the induced action by $G^1$.  

\subsection{Quasi-Semi-Simple Elements}
An element $g\in G$ is \textit{quasi-semi-simple} if it normalises a maximal torus $T$ and a Borel subgroup $B\subset G^{\circ}$ containing $T$ (\cite[D\'efinition 1.1]{DM94}). In other words, $g$ is contained in some group of the form $N_G(T,B)$. By \cite[Theorem 7.5]{St}, semi-simple elements are quasi-semi-simple. Note that the identity component of $N_G(T,B)$ is $T$, and  $N_G(T,B)$ meets all connected components of $G$, because all Borel subgroups and all maximal tori of $G^{\circ}$ are conjugate under $G^{\circ}$. A quasi-semi-simple element $s\in G$ is called \textit{quasi-central} if there is no $g\in G^{\circ}$ such that $\dim C_G(gs)>\dim C_G(s)$ (\cite[D\'efinition 1.15]{DM94}). For characteristic reasons (\cite[Remarque 2.7]{DM94}), all unipotent elements of $G$ are contained in $G^{\circ}$ and all quasi-semi-simple elements are semi-simple.

Now assume $G/G^{\circ}$ is a cyclic group and denote by $G^1$ a connected component generating the component group. Let $\sigma\in G^1$ be a semi-simple element, and let $T$ be a $\sigma$-stable maximal torus contained in a $\sigma$-stable Borel subgroup $B\subset G^{\circ}$. Denote by $[T,\sigma]$ the commutator and by $T^{\sigma}$ the centraliser $C_T(\sigma)$. The following proposition will be frequently used.

\begin{Prop}(\cite[Proposition 1.16]{DM18})\label{DM18Prop1.16}
Semi-simple classes in $G^1$ have representatives in $(T^{\sigma})^{\circ}\sigma$. Two elements $t\sigma$, $t'\sigma\in(T^{\sigma})^{\circ}\sigma$ represent the same $G^{\circ}$-conjugacy class if and only if $t$ and $t'$, when passing to the quotient $$T/[T,\sigma]\cong(T^{\sigma})^{\circ}/[T,\sigma]\cap (T^{\sigma})^{\circ},$$ belong to the same $W^{\sigma}$-orbit, where $W^{\sigma}$ is the subgroup of $W_{G^{\circ}}(T)$ consisting of the elements fixed by $\sigma$.
\end{Prop}

\subsection{Completely reducible and irreducible subgroups}\label{H(x)}
Let $G$ be a (not necessarily connected) reductive algebraic group. A closed subgroup $H$ of $G$ is \textit{$G$-completely reducible} if for any parabolic subgroup $P\subset G$ containing $H$, there is a Levi factor of $P$ containing $H$. A closed subgroup $H$ of $G$ is \textit{$G$-irreducible} if it is not contained in any subgroup of the form $N_G(P^{\circ})$ with $P^{\circ}\subset G^{\circ}$ being a proper parabolic subgroup. Clearly, if $G=G^{\circ}$, then the above definition coincides with the definitions for connected reductive groups. 
\begin{Rem}
Our definitions  of irreducible subgroup and completely reducible subgroup are slightly different from those in \cite{BMR}, which use R-parabolic subgroups and R-Levi subgroups. In this article, we only consider irreducible subgroups $H\subset G$ that meet all connected components of $G$. (These groups will appear as the image of $\rho$ in the diagram (\ref{rho-CD}).) For such a subgroup, our definitions agree with those in \cite{BMR}. An R-parabolic subgroup $P\subset G$ is always contained in $N_G(P^{\circ})$. If it meets all connected components of $G$, then it must be the entire normaliser group $N_G(P^{\circ})$. An R-Levi subgroup is just a usual Levi factor of an R-parabolic subgroup. Conversely, every subgroup of the form $N_G(P^{\circ})$ is an R-parabolic according to \cite[Proposition 2.4]{Ri88}.
\end{Rem}

Given $\mathbf{x}=(x_1,\ldots,x_n)\in G^n$, denote by $H(\mathbf{x})$ the closed algebraic subgroup of $G$ such that $H(\mathbf{x})(k)$ is the Zariski closure of the abstract group generated by $\{x_1,\ldots,x_n\}$. Consider the diagonal action of $G^{\circ}$ on $G^n$ by conjugation. Denote by $\Stab_{G^{\circ}}(\mathbf{x})$ the stabiliser. Obviously $\Stab_{G^{\circ}}(\mathbf{x})=\Stab_{G^{\circ}}(H(\mathbf{x}))$. Recall that for a $G^{\circ}$-action on an algebraic variety $X$, an orbit $G^{\circ}\bcdot x$, $x\in X$, is called \textit{stable}, if it is closed and $\Stab_{G^{\circ}}(x)/Z_X$ is finite,  where $Z_X:=\cap_{x\in X}\Stab_{G^{\circ}}(x)$.
\begin{Eg}
Suppose that $G^{\circ}=\GL_n$, and that a connected component $G^1$ of $G$ acts on $Z_{G^{\circ}}\cong\mathbb{G}_m$ by $t\mapsto t^{-1}$, so that $Z_G=\{\pm 1\}$. Let $X=G^n$, so $Z_X=Z_G$. If $H(\mathbf{x})$ is an irreducible subgroup contained in $G^{\circ}$, then $\Stab_{G^{\circ}}(\mathbf{x})=Z_{G^{\circ}}$. Therefore $\Stab_{G^{\circ}}(\mathbf{x})/Z_X$ is not finite and the orbit $G^{\circ}\cdot\mathbf{x}$ is not stable.
\end{Eg}

\begin{Thm}\label{st=irr}
Let $G$ and $\mathbf{x}$ be as above. Then the orbit $G^{\circ}\bcdot\mathbf{x}$ is closed if and only if $H(\mathbf{x})$ is completely reducible. Suppose that $H(\mathbf{x})$ meets all connected components of $G$. Then the $G^{\circ}$-orbit of $\mathbf{x}\in G^n$ is stable if and only if $H(\mathbf{x})$ is an irreducible subgroup of $G$.
\end{Thm}
\begin{proof}
The first part is \cite[Theorem 3.6]{Ri88}. Note that in \cite{Ri88}, $G$ being "linearly reductive" simply means that $G^{\circ}$ is reductive. The proof of \cite[Theorem 4.1]{Ri88} can be adapted to the current situation.
\end{proof}

\begin{Rem}
This theorem also holds in positive characteristic. The proof is a combination of \cite[\S 6.3]{BMR}, \cite[Proposition 8.3]{Ma} and \cite[Proposition 16.7]{Ri88}.
\end{Rem}

\subsection{$G\rtimes\Gamma$-Character Varieties}\label{GGAMMACV}

For any connected reductive group $G$ and a finitely generated discrete group $\Pi$, the set of homomorphisms $\Hom(\Pi,G)$ has an algebraic structure induced from that of $G$, and we define the $G$-representation variety $\Rep(\Pi,G)$ as this algebraic variety. Its elements are called $G$-representations of $\Pi$. The conjugation by $G$ on the target induces an action on $\Rep(\Pi,G)$. We will denote by $\Ch(\Pi,G)$ the resulting categorical quotient, called the $G$-character variety. 

Let $\Gamma$ be a finite quotient of $\Pi$ and fix a homomorphism $\psi:\Gamma\rightarrow\Aut G$, where $\Aut G$ is the group of automorphisms of $G$. Let $G\rtimes\Gamma$ be the semi-direct product defined by $\psi$. 
\begin{Defn}
We say that a homomorphism of groups $\rho:\Pi\rightarrow G\rtimes\Gamma$ is an \textit{admissible $G\rtimes\Gamma$-representation} of $\Pi$ if the right square of the following diagram commutes
\begin{equation}\label{rho-CD}
\begin{tikzcd}[row sep=2.5em, column sep=3em]
1 \arrow[r] & \tilde{\Pi} \arrow[r, "p"]\arrow[d, "\tilde{\rho}"] & \Pi \arrow[r]\arrow[d, "\rho"] & \Gamma \arrow[r]\arrow[d, "\Id"] & 1\\
1 \arrow[r] & G \arrow[r] & G\rtimes\Gamma \arrow[r] & \Gamma \arrow[r] & 1
\end{tikzcd}
\end{equation}
where $\tilde{\Pi}$ is the kernel and $\tilde{\rho}$ is just the restriction of $\rho$. We call $\tilde{\rho}$ the \textit{underlying $G$-representation} of $\rho$. (The left hand side of the diagram automatically commutes.)
\end{Defn}

Write $\bar{G}=G\rtimes\Gamma$. The set $\Rep_{\Gamma}(\Pi,G)$ of admissible $\bar{G}$-representations of $\Pi$ is an algebraic variety, which can be constructed as follows. Suppose that $\Pi_0$ is a free group with $n$ generators such that we have a surjective group homomorphism $\Pi_0\rightarrow\Pi$ with kernel $R$. Then each element $r\in R$ defines a closed subvariety of $\bar{G}^n$. As an example, let $\Pi$ be the fundamental group of a compact Riemann surface of genus $g$, then $n$ can be taken to be $2g$ and $R$ is generated by one single element $r=\prod_{i=1}^g[\alpha_i,\beta_i]$, where $\alpha_i$'s and $\beta_i$'s are the generators of the free group $\Pi_0$. Then $r$ defines a closed subvariety $$\{(A_i,B_i)_i\in \bar{G}^{2g}|\prod_{i=1}^g[A_i,B_i]=1\}\subset\bar{G}^{2g}.$$ The admissibility of $\bar{G}$-representations amounts to requiring the images of the generators of $\Pi$ to lie in some specific connected components of $\bar{G}$ according to (\ref{rho-CD}). In the general case, we define $\Rep_{\Gamma}(\Pi,G)$ as the intersection of all these subvarieties as $r$ runs over $R$. The conjugation by $G$ on $G\rtimes\Gamma$ induces an action of $G$ on $\Rep_{\Gamma}(\Pi,G)$, and we will denote by $\Ch_{\Gamma}(\Pi,G)$ the categorical quotient, called the $G\rtimes\Gamma$-character variety.  

Let $\rho:\Pi\rightarrow\bar{G}$ be an admissible $\bar{G}$-representation, let $\mathbf{x}$ be a tuple of elements of $\bar{G}$ which are images of a finite set of generators of $\Pi$, and let $H(\mathbf{x})$ be as defined in \S \ref{H(x)}. We say that an admissible $\rho$ is \textit{semi-simple} if $H(\mathbf{x})$ is a completely reducible subgroup of $\bar{G}$. We say that an admissible $\rho$ is \textit{irreducible} if $H(\mathbf{x})$ is an irreducible subgroup of $\bar{G}$. Let $\tilde{\rho}$ be a $G$-representation of $\tilde{\Pi}$. We say that $\tilde{\rho}$ is \textit{strongly irreducible} if $\tilde{\rho}$ is irreducible and $\Stab_{G}(\tilde{\rho})=Z_G$.

\begin{Thm}
The $G$-orbit of $\rho\in\Rep_{\Gamma}(\Pi,G)$ is closed if and only if $\rho$ is a semi-simple $\bar{G}$-representation of $\Pi$. The $G$-orbit of $\rho\in\Rep_{\Gamma}(\Pi,G)$ is stable if and only if $\rho$ is an irreducible $\bar{G}$-representation of $\Pi$.
\end{Thm}
\begin{proof}
By Theorem \ref{st=irr}, the assertions hold when $\Pi$ is a free group with $n$ generators. Since our representation variety can be realised as a closed $G$-invariant subvariety of $\bar{G}^n$ for some $n$, we are done.
\end{proof}

\subsection{The Group $G\rtimes\Gamma$.}\label{GGAMMA}
We give the classification of the semi-direct products $G\rtimes\Gamma$, with fixed $G$ and $\Gamma$. Our reference for group cohomology is \cite[Chapitre I, \S 5]{Ser}. 

Denote by $Z_G$ the centre of $G$ and denote by $\Inn G=G/Z_G$ the group of inner automorphisms, which is a normal subgroup of $\Aut G$.  Denote by $\Out(G)$ the quotient group $\Aut G/\Inn G$. Let $\Gamma$ be a discrete group. For any homomorphism $\psi:\Gamma\rightarrow\Aut G$, denote by $\Omega(\psi)$ its composition with the quotient map $\Aut G\rightarrow\Out G$.

A homomorphism $\psi:\Gamma\rightarrow\Aut G$ defines a semi-direct product $G\rtimes_{\psi}\Gamma$, which will often be written as $G\rtimes\Gamma$. Each semi-direct product is equipped with a natural section $s:\Gamma\rightarrow G\rtimes\Gamma$, which is a group homomorphism, satisfying
$$
\psi_{\sigma}(g)=s_{\sigma}gs_{\sigma}^{-1},\text{ for any $g\in G$ and $\sigma\in\Gamma$},
$$
where we write $s_{\sigma}=s(\sigma)$ and $\psi_{\sigma}=\psi(\sigma)$ for any $\sigma\in\Gamma$.

\begin{Lem}
Let $H$ be a group. Let $s:\Gamma\rightarrow G\rtimes_{\psi}\Gamma$ be the natural section. Let $\Phi:G\rtimes_{\psi}\Gamma\rightarrow H$ be a map of set. Suppose that the restriction of $\Phi$ to $G$ is a group homomorphism and that 
\begin{itemize}\label{gp-hom}
\item[-] $\Phi(gs_{\sigma})=\Phi(g)\Phi(s_{\sigma})$, for any $g\in G$ and $\sigma\in\Gamma$;
\item[-] $\Phi(s_{\sigma})\Phi(s_{\tau})=\Phi(s_{\sigma\tau})$;
\item[-] $\Phi(s_{\sigma})\Phi(g)\Phi(s_{\sigma})^{-1}=\Phi(\psi_{\sigma}(g))$, for any $\sigma$, $\tau\in\Gamma$.
\end{itemize} 
Then $\Phi$ is a group homomorphism.
\end{Lem}
\begin{proof}
Obvious.
\end{proof}

\begin{Lem}\label{f-isom}
Let $\psi\in\Hom(\Gamma,\Aut G)$. Let $\Phi_0\in\Aut G$ and let $f\in\Aut\Gamma$. Write $\tilde{\psi}=\ad\Phi_0\circ\psi\circ f$, where $\ad\Phi_0:\Aut G\isom\Aut G$ is the conjugation by $\Phi_0$. Then there is an isomorphism $\Phi:G\rtimes_{\psi}\Gamma\isom G\rtimes_{\tilde{\psi}}\Gamma$ that coincides with $f^{-1}$ modulo $G$ and restricts to $\Phi_0$ on $G$.
\end{Lem}
\begin{proof}
Let $s:\Gamma\rightarrow G\rtimes_{\psi}\Gamma$ and $\tilde{s}:\Gamma\rightarrow G\rtimes_{\tilde{\psi}}\Gamma$ be the natural sections. For any $g\in G$, we have $\tilde{s}_{\sigma}g\tilde{s}_{\sigma}^{-1}=\Phi_0\circ\psi_{\sigma}\circ\Phi_0^{-1}(g)$. For any $g\in G$ and $\sigma\in\Gamma$, define $\Phi(gs_{\sigma})=\Phi_0(g)\tilde{s}_{f^{-1}(\sigma)}$. This is obviously a bijection. It suffices to verify that the conditions of Lemma \ref{gp-hom} are satisfied.
\end{proof}

\begin{Prop}\label{prop-exist-isom}
Let $\psi$, $\psi'\in\Hom(\Gamma,\Aut G)$. Then there is an isomorphism $\Phi:G\rtimes_{\psi}\Gamma\isom G\rtimes_{\psi'}\Gamma$ if and only if there exist $f\in\Aut\Gamma$ and $\Phi_0\in\Aut G$ such that the following conditions hold. 

(i) $\Omega(\psi)=\Omega(\ad\Phi_0\circ\psi'\circ f^{-1})$.

(ii) Write $\psi''=\ad\Phi_0\circ\psi'\circ f^{-1}$. There exists an isomorphism $\Phi':G\rtimes_{\psi}\Gamma\isom G\rtimes_{\psi''}\Gamma$ satisfying:
\begin{itemize}
\item[-] $\Phi'|_G=\Id$;
\item[-] $\Phi'\equiv\Id\mod G$.
\end{itemize}
\end{Prop}
\begin{proof}
If $f$ and $\Phi_0$ exist, then $\Phi'$ composed with the isomorphism in Lemma \ref{f-isom} gives the desired isomorphism.

Conversely, suppose that there is an isomorphism $\Phi:G\rtimes_{\psi}\Gamma\isom G\rtimes_{\psi'}\Gamma$. Let $f$ be the automorphism of $\Gamma$ induced by $\Phi$ on the component group. Applying Lemma \ref{f-isom} to $f$ and $\Phi_0=\Id$, we get an isomorphism $\tilde{\Phi}':G\rtimes_{\psi'}\Gamma\isom G\rtimes_{\tilde{\psi}'}\Gamma$, where $\tilde{\psi}'=\psi\circ f$. Put $\tilde{\Phi}=\tilde{\Phi}'\circ\Phi$, then $\tilde{\Phi}$ induces the identity map on the component group. Denote by $s:\Gamma\rightarrow G\rtimes_{\psi}\Gamma$ and $\tilde{s}':G\rtimes_{\tilde{\psi}'}\Gamma$ the corresponding natural sections. For any $\sigma\in\Gamma$, we can write $\tilde{\Phi}(s_{\sigma})=g_{\sigma}\tilde{s}'_{\sigma}$ for some $g_{\sigma}\in G$. We apply $\tilde{\Phi}$ to $s_{\sigma}gs_{\sigma}^{-1}$ and get $$\tilde{\Phi}(\psi_{\sigma}(g))=g_{\sigma}\tilde{\psi}'_{\sigma}(\tilde{\Phi}(g))g^{-1}_{\sigma}$$ for any $g\in G$. This shows that $\psi_{\sigma}$ is equal to $\tilde{\Phi}^{-1}\circ\tilde{\psi}'_{\sigma}\circ\tilde{\Phi}$ up to an inner automorphism. Now let $\Phi_0$ be the automorphism of $G$ obtained by restricting $\tilde{\Phi}$ to $G$, then apply Lemma \ref{f-isom} to $\Phi_0^{-1}$ and $f=\Id$.
\end{proof}
\begin{Rem}\hfill

(i). In the proof of the above proposition we use the fact that $G$ is a connected group so that any isomorphism $\Phi$ must map $G$ to $G$.

(ii). The isomorphisms $\Phi'$ in the statement of the proposition can be characterised as follows.\footnote{I thank a referee for suggesting this to me.} Let $s'':\Gamma\rightarrow G\rtimes_{\psi''}\Gamma$ be the natural section. For any $\sigma\in\Gamma$, write $\Phi'(s_{\sigma})=h_{\sigma}^{-1}s''_{\sigma}$ for some $h_{\sigma}\in G$. Then $(h_{\sigma})_{\sigma\in\Gamma}$ defines a cocycle. Conversely, let $(h_{\sigma})_{\sigma}\in(\prod_{\sigma\in\Gamma}G)$ be a cocycle and put $\tilde{\psi}_{\sigma}:=(\ad h_{\sigma})\psi_{\sigma}$ for any $\sigma\in\Gamma$. Then $\tilde{\psi}$ is a group homomorphism from $\Gamma$ to $\Aut G$ with $\Omega(\tilde{\psi})=\Omega(\psi)$. Thus we have a semi-direct product $G\rtimes_{\tilde{\psi}}\Gamma$ with natural section $\tilde{s}:\Gamma\rightarrow G\rtimes_{\tilde{\psi}}\Gamma$. It can be checked that $\Phi_{(h_{\sigma})_{\sigma}}:G\rtimes_{\psi}\Gamma\rightarrow G\rtimes_{\tilde{\psi}}\Gamma$, which sends $gs_{\sigma}$ to $gh_{\sigma}^{-1}\tilde{s}_{\sigma}$, defines an isomorphism of groups satisfying the conditions in (ii) of the above proposition. Therefore, the isomorphisms $\Phi'$ are necessarily of the form $\Phi_{(h_{\sigma})_{\sigma}}$.
\end{Rem}

Let $\psi$ and $\psi'\in\Hom(\Gamma,\Aut G)$ be such that $\Omega(\psi')=\Omega(\psi)$. We say that $\psi$ and $\psi'$ are equivalent if there exists $g\in G$ such that $\psi'_{\sigma}=(\ad g)\psi_{\sigma}(\ad g^{-1})$ for all $\sigma\in\Gamma$, where $\ad g$ is the inner automorphism of $G$ defined by $g$.
\begin{Lem}\label{equiv-psi}
Suppose that $\Omega(\psi')=\Omega(\psi)$ and that $\psi'$ is equivalent to $\psi$. Then there is an isomorphism $\Phi:G\rtimes_{\psi}\Gamma\isom G\rtimes_{\psi'}\Gamma$ satisfying 
\begin{itemize}
\item[-] $\Phi|_G=\Id$;
\item[-] $\Phi\equiv\Id\mod G$.
\end{itemize}
\end{Lem}
\begin{proof}
Let $g\in G$ be such that $\psi'=(\ad g)\psi(\ad g^{-1})$. Let $s$ and $s'$ be the natural sections. For each $\sigma\in\Gamma$, put $g_{\sigma}=\psi_{\sigma}(g)g^{-1}$, and define $\Phi(xs_{\sigma})=xg_{\sigma}s'_{\sigma}$, for any $x\in G$ and $\sigma\in\Gamma$. This is obviously a bijection. Unwinding the definitions, we see that the conditions of Lemma \ref{gp-hom} are satisfied.
\end{proof}

\begin{Lem}\label{equiv-of-inner}
Fix $\psi\in\Hom(\Gamma,\Aut G)$. Then the equivalence classes of those $\psi'$ such that $\Omega(\psi')=\Omega(\psi)$ are parametrised by the pointed set $H^1(\Gamma,\Inn G)$ with the equivalence class of $\psi$ corresponding to the trivial cohomology class.
\end{Lem}
In what follows we will identify these equivalence classes with $H^1(\Gamma,\Inn G)$ and regard $\psi$ as the base point.
\begin{proof}
The group of inner automorphisms $\Inn G$ is a normal subgroup of $\Aut G$. For any $\sigma\in\Gamma$, $\psi_{\sigma}$ acts by conjugation on $\Inn G$. Recall that $(z_\sigma)_{\sigma\in\Gamma}\in(\Inn G)^{|\Gamma|}$ is a cocycle if $z_{\sigma}\psi_{\sigma}(z_{\tau})=z_{\sigma\tau}$ for any $\sigma$, $\tau\in\Gamma$. Two cocycles $(z_{\sigma})_{\sigma}$ and $(z'_{\sigma})_{\sigma}$ are equivalent if there exists $y\in\Inn G$ such that $z'_{\sigma}=yz_{\sigma}\psi_{\sigma}(y)^{-1}$ for all $\sigma\in\Gamma$. Then $H^1(\Gamma,\Inn G)$ is the set of equivalence classes of cocycles. It is a pointed set with the base point being the class of the trivial cocycle.

If $\Omega(\psi')=\Omega(\psi)$, then $\psi'_{\sigma}=x_{\sigma}\psi_{\sigma}$, for some $x_{\sigma}\in \Inn G$ and each $\sigma\in\Gamma$. Since both $\psi$ and $\psi'$ are group homomorphisms, we deduce that $x_{\sigma}\psi_{\sigma}(x_{\tau})=x_{\sigma\tau}$ for any $\sigma$, $\tau\in\Gamma$. Therefore $(x_{\sigma})_{\sigma}$ defines a cocycle. It is easy to check that the equivalence relation of homomorphisms translates into the equivalence relation of cocycles.
\end{proof}
By Lemma \ref{equiv-psi} and Lemma \ref{equiv-of-inner}, to each element of $H^1(\Gamma,\Inn G)$ is associated an isomorphism class of $G\rtimes\Gamma$. It remains to determine whether two elements of $H^1(\Gamma,\Inn G)$ define isomorphic semi-direct products, with an isomorphism satisfying the assumptions of Proposition \ref{prop-exist-isom} (ii).
\begin{Lem}\label{imageofdelta}
Fix $\psi\in\Hom(\Gamma,\Aut G)$ and let $H^1(\Gamma,\Inn G)$ be defined by $\psi$ so that it has $\psi$ as the base point. Let $\psi'$ and $\psi''\in\Hom(\Gamma,\Aut G)$ be such that $\Omega(\psi')=\Omega(\psi'')=\Omega(\psi)$. We denote their equivalence classes also by $\psi'$ and $\psi''$ respectively. Let $\delta:H^1(\Gamma,\Inn G)\rightarrow H^2(\Gamma,Z_G)$ be the natural map.  Then there is an isomorphism $\Phi:G\rtimes_{\psi'}\Gamma\isom G\rtimes_{\psi''}\Gamma$ satisfying
\begin{itemize}
\item[-] $\Phi|_G=\Id$;
\item[-] $\Phi\equiv\Id\mod G$.
\end{itemize}
if and only if $\delta(\psi')=\delta(\psi'')$.
\end{Lem}
\begin{proof}
Suppose that $\psi'_{\sigma}=(\ad x_{\sigma})\psi_{\sigma}$ and $\psi''_{\sigma}=(\ad y_{\sigma})\psi_{\sigma}$ for some $x_{\sigma}$, $y_{\sigma}\in G$ and each $\sigma\in\Gamma$. Then $\delta(\psi')$ is by definition the cohomology class of the cocycle $(d_{\sigma\tau})_{\sigma,\tau\in\Gamma}$, with $$d_{\sigma\tau}:=x^{-1}_{\sigma\tau}x_{\sigma}\psi_{\sigma}(x_{\tau}).$$ The equality $\delta(\psi')=\delta(\psi'')$ is equivalent to the existence of a cochain $(c_{\sigma})_{\sigma\in\Gamma}$, $c_{\sigma}\in Z_G$, such that $$x^{-1}_{\sigma\tau}x_{\sigma}\psi_{\sigma}(x_{\tau})\cdot c^{-1}_{\sigma\tau}c_{\sigma}\psi_{\sigma}(c_{\tau})=y^{-1}_{\sigma\tau}y_{\sigma}\psi_{\sigma}(y_{\tau}),$$which is equivalent to
\begin{equation}\label{eq-xyc}
y_{\sigma\tau}x^{-1}_{\sigma\tau}c^{-1}_{\sigma\tau}=(y_{\sigma}x_{\sigma}^{-1}c^{-1}_{\sigma})x_{\sigma}\psi_{\sigma}(y_{\tau}x^{-1}_{\tau}c^{-1}_{\tau})x^{-1}_{\sigma}.
\end{equation}
For each $\sigma\in\Gamma$, put $z_{\sigma}=y_{\sigma}x_{\sigma}^{-1}c^{-1}_{\sigma}\in G$. Note that $z_1$ always equals to $1$. Denote by $s'$ and $s''$ the natural sections of $G\rtimes_{\psi'}\Gamma$ and $G\rtimes_{\psi''}\Gamma$ respectively. Define a bijection $\Phi(hs'_{\sigma}):=hz^{-1}_{\sigma}s''_{\sigma}$, for any $\sigma\in\Gamma$, $h\in G$. With these definitions, the conditions of Lemma \ref{gp-hom} are satisfied:
\begin{itemize}
\item[(i)] $\Phi(s'_{\sigma})\Phi(s'_{\tau})=\Phi(s'_{\sigma\tau})$;
\item[(ii)] $\Phi(s'_{\sigma})g\Phi(s'_{\sigma})^{-1}=\psi'_{\sigma}(g)$, for any $\sigma$, $\tau\in\Gamma$.
\end{itemize} 
Thus $\Phi$ is indeed an isomorphism.

Conversely, suppose that there is an isomorphism $\Phi:G\rtimes_{\psi'}\Gamma\isom G\rtimes_{\psi''}\Gamma$ satisfying the two assumptions. Write $\Phi(s'_{\sigma})=z^{-1}_{\sigma}s''_{\sigma}$ for some $z_{\sigma}\in G$ and each $\sigma\in\Gamma$. Then condition (ii) above is equivalent to $z_{\sigma}=y_{\sigma}x_{\sigma}^{-1}c^{-1}_{\sigma}$ for some $c_{\sigma}\in Z_G$ and each $\sigma\in\Gamma$. Condition (i) above is equivalent to $$z_{\sigma}x_{\sigma}\psi_{\sigma}(z_{\tau})x_{\sigma}^{-1}=z_{\sigma\tau}.$$These two conditions combined recover (\ref{eq-xyc}).
\end{proof}

Consider the left action of $\Aut G\times\Aut\Gamma$ on $\Hom(\Gamma,\Out G)$ by $$(\Phi_0,f):\Psi\mapsto\ad\Phi_0\circ\Psi\circ f^{-1},$$ regarding $\Phi_0$ as an element of $\Out G$. Denote by $\mathbf{A}(\Gamma,G)$ the set of $\Aut G\times\Aut\Gamma$-orbits under this action. For any $\Psi\in\Hom(\Gamma,\Out G)$, denote by $\mathbb{S}(\Psi)$ the stabiliser of $\Psi$ in $\Aut G\times\Aut\Gamma$. Let $H^1(\Gamma,\Inn G)$ be defined by some $\psi$ such that $\Omega(\psi)=\Psi$. Then $\mathbb{S}(\Psi)$ acts on $H^1(\Gamma,\Inn G)$, sending $\psi'$ to $\ad\Phi_0\circ\psi'\circ f^{-1}$. By Lemma \ref{f-isom} and Lemma \ref{imageofdelta}, it descends to an action on $\Ima\delta$. Denote by $\mathbf{B}(\Gamma,G,\Psi)$ the set of $\mathbb{S}$-orbits in $\Ima\delta$.

\begin{Thm}\label{para-GGamma}
For each $a\in\mathbf{A}(\Gamma,G)$, we choose a representative $\Psi_a\in\Hom(\Gamma,\Out G)$. Then the isomorphism classes of the semi-direct products $G\rtimes\Gamma$ are parametrised by the pairs $(a,b)$, with $a\in\mathbf{A}(\Gamma,G)$ and $b\in\mathbf{B}(\Gamma,G,\Psi_a)$.
\end{Thm}
\begin{proof}
This is a combination of Proposition \ref{prop-exist-isom}, Lemma \ref{equiv-of-inner} and Lemma \ref{imageofdelta}.
\end{proof}

\section{$\Ch_{\Gamma}(\Pi,G)$ and $\Ch(\tilde{\Pi},G)^{\Gamma}$}\label{GAMMAINV}
In this section,  we investigate the relation between $\Ch_{\Gamma}(\Pi,G)$ and $\Gamma$-fixed points in $\Ch(\tilde{\Pi},G)$. 

\subsection{}
Let $\Pi$ be a finitely generated discrete group and let $p:\tilde{\Pi}\rightarrow\Pi$ be a finitely generated normal subgroup with finite index, i.e. we have the short exact sequence
$$
1\longrightarrow\tilde{\Pi}\stackrel{p}{\longrightarrow}\Pi\longrightarrow \Gamma\longrightarrow 1.
$$
We choose once and for all a section $\gamma_{\ast}:\Gamma\rightarrow\Pi$. In general, it is only a map of sets, but we can always require that $\gamma_1=1$. We will write $\gamma_{\sigma}=\gamma_{\ast}(\sigma)$ for $\sigma\in\Gamma$.

Fix a group homomorphism $\psi:\Gamma\rightarrow \Aut G$.

\begin{Defn}
For any $\tilde{\rho}\in\Rep(\tilde{\Pi},G)$, we say $\tilde{\rho}$ is \textit{pre-$(\Gamma,\psi)$-invariant} if there exists some cochain $h_{\ast}=(h_{\sigma})_{\sigma\in\Gamma}\in C^1(\Gamma,G)$ such that for any $\sigma\in\Gamma$,
\begin{equation}\label{sigma-rho}
\ad h_{\sigma}\circ\tilde{\rho}=\leftidx{^{\sigma}\!}{\tilde{\rho}}:=\psi_{\sigma}\circ\tilde{\rho}\circ C_{\sigma},
\end{equation}
where $C_{\sigma}$ is the automorphism
\begin{equation}
\begin{split}
\tilde{\Pi}&\lisom\tilde{\Pi}\\
\alpha&\longmapsto \gamma^{-1}_{\sigma}\alpha\gamma_{\sigma}.
\end{split}
\end{equation} 
In this case, we say that $(\tilde{\rho},h_{\ast})$ is a \textit{pre-$(\Gamma,\psi)$-invariant pair}.
\end{Defn}
\begin{Defn}
Let $(\tilde{\rho},h_{\ast})$ be a pre-$(\Gamma,\psi)$-invariant pair. For any $\sigma_1$, $\sigma_2\in\Gamma$, put $g_{\sigma_1\sigma_2}:=\tilde{\rho}(\gamma_{\sigma_1}\gamma_{\sigma_2}\gamma_{\sigma_1\sigma_2}^{-1})$. We say that $(\tilde{\rho},h_{\ast})$ is \textit{$(\Gamma,\psi)$-invariant} if $h_1=1$ and the equality 
\begin{equation}\label{cocycle-rho}
g_{\sigma_1\sigma_2}=h_{\sigma_1}^{-1}\psi_{\sigma_1}(h^{-1}_{\sigma_2})h_{\sigma_1\sigma_2}
\end{equation} holds for any $\sigma_1$, $\sigma_2\in\Gamma$.
\end{Defn}

We will simplify the notation in what follows by writing $\ad h_{\sigma}\circ\tilde{\rho}$ as $h_{\sigma}\tilde{\rho}$, which is more natural from the point of view of $G$-action on $\Rep(\tilde{\Pi},G)$. However, when we evaluate $\tilde{\rho}$ at a particular element, say $\alpha$, we will use the usual notation $h_{\sigma}\tilde{\rho}(\alpha)h_{\sigma}^{-1}$.

\begin{Rem}\label{GammaAction}
Note that $\tilde{\rho}\mapsto\leftidx{^{\sigma}\!}{\tilde{\rho}}$ does not define a group action of $\Gamma$ on $\Rep(\tilde{\Pi},G)$ because $\gamma_{\tau}\gamma_{\sigma}$ is not necessarily equal to $\gamma_{\tau\sigma}$, for any $\sigma$, $\tau\in\Gamma$. But since $\tilde{\rho}(\gamma_{\tau}\gamma_{\sigma}\gamma_{\tau\sigma}^{-1})\in G$, we do have an action of $\Gamma$ on $\Ch(\tilde{\Pi},G)$.
\end{Rem} 

\subsection{}
Recall that an admissible $G\rtimes\Gamma$-representation $\rho:\Pi\rightarrow G\rtimes\Gamma$ restricts to a homomorphism $\tilde{\rho}:\tilde{\Pi}\rightarrow G$, which is called the underlying $G$-representation. Also recall that the semi-direct product $G\rtimes\Gamma$ is equipped with a natural section $s:\Gamma\rightarrow G\rtimes\Gamma$. 

Lemma \ref{resofrho} and Lemma \ref{extofinv} below will give a bijection between $(\Gamma,\psi)$-invariant pairs $(\tilde{\rho},h_{\ast})$ and admissible $G\rtimes\Gamma$-representations.
\begin{Lem}\label{resofrho}
If $\tilde{\rho}$ is the underlying $G$-representation of some admissible $G\rtimes\Gamma$-representation $\rho$, then there exists some cochain $h_{\ast}$ such that $(\tilde{\rho},h_{\ast})$ is $(\Gamma,\psi)$-invariant.
\end{Lem}
\begin{proof}
By the commutativity of the diagram (\ref{rho-CD}), we have $s_{\sigma}\rho(\gamma_{\sigma})^{-1}\in G$ for any $\sigma\in\Gamma$. So we define $h_{\sigma}:=s_{\sigma}\rho(\gamma_{\sigma})^{-1}$. We calculate, for any $\alpha\in\tilde{\Pi}$,
\begin{equation}
\begin{split}
\leftidx{^{\sigma}\!}{\tilde{\rho}}(\alpha)&=\psi_{\sigma}\circ\rho\circ p \circ C_{\sigma}(\alpha)\\
&=\psi_{\sigma}\circ\rho(\gamma_{\sigma}^{-1}\alpha\gamma_{\sigma})\\
&=s_{\sigma}\left(\rho(\gamma_{\sigma}^{-1})\tilde{\rho}(\alpha)\rho(\gamma_{\sigma})\right)s_{\sigma}^{-1}\\
&=h_{\sigma}\tilde{\rho}(\alpha)h_{\sigma}^{-1}.
\end{split}
\end{equation}
The equality (\ref{cocycle-rho}) is then equivalent to 
$$
\rho(\gamma_{\sigma_1}\gamma_{\sigma_2}\gamma_{\sigma_1\sigma_2}^{-1})=\rho(\gamma_{\sigma_1})\rho(\gamma_{\sigma_2})\rho(\gamma_{\sigma_1\sigma_2})^{-1},
$$which is obvious.
\end{proof}

\subsection{}
Conversely, we want to extend a $(\Gamma,\psi)$-invariant $(\tilde{\rho},h_{\ast})$ to an admissible  $G\rtimes\Gamma$-representation $\rho$. The proof of the above lemma suggests defining
\begin{equation}\label{extFML1}
\rho(\gamma_{\sigma}):=h_{\sigma}^{-1}s_{\sigma}.
\end{equation}
Any element of $\Pi$ can be uniquely written as $\eta\gamma_{\sigma}$ for some $\sigma\in\Gamma$ and $\eta\in\tilde{\Pi}$. We then define 
\begin{equation}\label{extFML2}
\rho(\eta\gamma_{\sigma}):=\tilde{\rho}(\eta)\rho(\gamma_{\sigma}).
\end{equation}
In particular, $\rho|_{\tilde{\Pi}}=\tilde{\rho}$.
\begin{Lem}\label{extofinv}
Let $(\tilde{\rho},h_{\ast})$ be a pre-$(\Gamma,\psi)$-invariant pair. Then,
\begin{itemize}
\item[(i)]
The equality $$g_{\sigma_1\sigma_2}=h_{\sigma_1}^{-1}\psi_{\sigma_1}(h^{-1}_{\sigma_2})h_{\sigma_1\sigma_2}$$ holds up to multiplication by $\Stab_G(\tilde{\rho})$ on the left;
\item[(ii)]
The formulae (\ref{extFML1}) and (\ref{extFML2}) define a homomorphism of groups $\Pi\rightarrow G\rtimes\Gamma$ if and only if $(\tilde{\rho},h_{\ast})$ is $(\Gamma,\psi)$-invariant.
\end{itemize}
\end{Lem}
In fact, $\Stab_G(\tilde{\rho})$ commutes with $g_{\sigma_1\sigma_2}$ which by definition lies in the image of $\tilde{\rho}$.
\begin{proof}
(i). Write $\sigma=\sigma_1\sigma_2$. Let us compute $\Phi:=\leftidx{^{\sigma_1}}{}(\leftidx{^{\sigma_2}}{}(\psi^{-1}_{\sigma}\circ\tilde{\rho}\circ C_{\sigma}^{-1}))$. Note that the equality $\psi_{\sigma}\circ\tilde{\rho}\circ C_{\sigma}=h_{\sigma}\tilde{\rho}$ implies $\psi_{\sigma}^{-1}\circ\tilde{\rho}\circ C_{\sigma}^{-1}=\psi_{\sigma}^{-1}(h^{-1}_{\sigma})\tilde{\rho}$.

On the one hand, using the equality $\leftidx{^{\tau}\!}{\tilde{\rho}}=\psi_{\tau}\circ\tilde{\rho}\circ C_{\tau}$ for $\tau$ equal to $\sigma_1$ and $\sigma_2$, we have,
\begin{equation}
\Phi=\tilde{\rho}\circ C_{\sigma}^{-1}\circ C_{\sigma_2}\circ C_{\sigma_1}.
\end{equation}
Since $g_{\sigma_1\sigma_2}=\tilde{\rho}(\gamma_{\sigma_1}\gamma_{\sigma_2}\gamma_{\sigma}^{-1})$, the above equation gives $\Phi=g_{\sigma_1\sigma_2}^{-1}\tilde{\rho}$.

On the other hand, using the equality $\leftidx{^{\tau}\!}{\tilde{\rho}}=h_{\tau}\tilde{\rho}$ for $\tau$ equal to $\sigma_1$, $\sigma_2$ and the equality at the beginning of the proof, we have,
\begin{equation}
\Phi=h_{\sigma}^{-1}\psi_{\sigma_1}(h_{\sigma_2})h_{\sigma_1}\tilde{\rho}.
\end{equation}
whence the first part.

(ii). Since $\gamma_1=1$ and $s_{1}=1$, the equality (\ref{extFML1}) implies that $h_1$ must be $1$. The $(\Gamma,\psi)$-invariant property implies
$$
\rho(\gamma_{\sigma_1}\gamma_{\sigma_2}\gamma_{\sigma}^{-1})=\rho(\gamma_{\sigma_1})\rho(\gamma_{\sigma_2})\rho(\gamma_{\sigma})^{-1},
$$
for any $\sigma_1$, $\sigma_2\in\Gamma$. One checks by direct computation that this equality implies $\rho(\eta_1\gamma_{\sigma_1}\eta_2\gamma_{\sigma_2})=\rho(\eta_1\gamma_{\sigma_1})\rho(\eta_2\gamma_{\sigma_2})$, for any $\sigma_1$, $\sigma_2\in\Gamma$, and any $\eta_1$, $\eta_2\in \tilde{\Pi}$.
\end{proof}

\subsection{Remarks on Fundamental Groups}\label{RemFunGrps}
Let $p:\tilde{X}\rightarrow X$ be a covering of topological spaces with Galois group $\Gamma$ (the group of covering transformations $\Aut(\tilde{X}/X)$). Choose base points $\tilde{x}\in\tilde{X}$ and $x=p(\tilde{x})$. Our convention is that by a juxtaposition $\beta\alpha$ of paths we mean the path starting from $\alpha$ and ending along $\beta$, so that we have the short exact sequence
\begin{equation}
1\longrightarrow\pi_1(\tilde{X},\tilde{x})\longrightarrow\pi_1(X,x)\longrightarrow \Gamma^{op}\longrightarrow 1.
\end{equation}
Then the general arguments for discrete groups apply to $\tilde{\Pi}=\pi_1(\tilde{X})$ and $\Pi=\pi_1(X)$. We choose a section $\gamma_{\ast}=(\gamma_{\sigma})_{\sigma\in\Gamma}$ of the natural map $\pi_1(X)\rightarrow\Gamma^{op}$ as in the general setting. 

Let $\lambda_{\sigma}$ be the unique lift of $\gamma_{\sigma}$ starting from $\tilde{x}$. For any $\sigma\in\Gamma$, let $\sigma$ also denote the isomorphism $\pi_1(\tilde{X},\tilde{x})\rightarrow\pi_1(\tilde{X},\sigma(\tilde{x}))$ and denote by $C_{\lambda_{\sigma}}$ the isomorphism
\begin{equation}
\begin{split}
\pi_1(\tilde{X},\sigma(\tilde{x}))&\longrightarrow\pi_1(\tilde{X},\tilde{x})\\
\alpha&\longmapsto \lambda^{-1}_{\sigma}\alpha\lambda_{\sigma},
\end{split}
\end{equation} 
For any $\alpha\in\pi_1(\tilde{X},\tilde{x})$, $C_{\lambda_{\sigma}}\circ\sigma(\alpha)=\lambda^{-1}_{\sigma}\sigma(\alpha)\lambda_{\sigma}$ is the unique lift of $\gamma_{\sigma}^{-1}\alpha\gamma_{\sigma}\in\pi_1(X)$ in $\pi_1(\tilde{X})$, therefore $C_{\lambda_{\sigma}}\circ\sigma$ can be identified with the conjugation by $\gamma_{\sigma}^{-1}$. Now pre-$(\Gamma,\psi)$-invariant $G$-representation should be defined by
\begin{equation}
\ad h_{\sigma}\circ\tilde{\rho}=\leftidx{^{\sigma}\!}{\tilde{\rho}}{}:=\psi_{\sigma}^{-1}\circ\tilde{\rho}\circ C_{\lambda_{\sigma}}\circ \sigma,
\end{equation}The reason why we have $\psi_{\sigma}^{-1}$ instead of $\psi_{\sigma}$ is as follows. Let $\psi^{op}:\Gamma^{op}\rightarrow \Aut G$ be the composition of $\psi$ and $\Gamma^{op}\rightarrow\Gamma$, $x\mapsto x^{-1}$. It defines a semi-direct product $G\rtimes_{\psi^{op}}\Gamma^{op}$, which is equipped with a section $s:\Gamma^{op}\rightarrow G\rtimes_{\psi^{op}}\Gamma^{op}$. In this circumstance, the section $s$ satisfies
\begin{equation}\label{spsi}
\psi_{\sigma^{-1}}(g)=\psi^{op}_{\sigma}(g)=s_{\sigma}gs_{\sigma}^{-1}.
\end{equation} 

In case $\tilde{\Pi}=\pi_1(\tilde{X})$ and $\Pi=\pi_1(X)$ are fundamental groups of some topological spaces, we may write $\Rep(\tilde{X},G)$, $\Rep_{\Gamma}(X,G)$, $\Ch(\tilde{X},G)$ and $\Ch_{\Gamma}(X,G)$.

\subsection{}
Fix $\Psi\in\Hom(\Gamma,\Out G)$ and let $\psi\in\Hom(\Gamma,\Aut G)$ be such that $\Omega(\psi)=\Psi$. As always, we choose a section $\gamma_{\ast}:\Gamma\rightarrow\Pi$ with $\gamma_1=1$. According to Remark \ref{GammaAction}, there is a $\Gamma$-action on $\Ch(\tilde{\Pi},G)$. Denote by $\Ch(\tilde{\Pi},G)^{\Gamma}$ the subvariety of $\Gamma$-fixed points. The action of $\Gamma$ and the fixed points locus $\Ch(\tilde{\Pi},G)^{\Gamma}$ only depend on $\Psi$.
\begin{Lem}
Let $\tilde{\rho}$ be a semi-simple $G$-representation of $\tilde{\Pi}$ and denote by $[\tilde{\rho}]$ its $G$-orbit. Then there exists some cochain $h_{\ast}$ such that $(\tilde{\rho},h_{\ast})$ is a pre-$(\Gamma,\psi)$-invariant pair if and only if $[\tilde{\rho}]$ lies in $\Ch(\tilde{\Pi},G)^{\Gamma}$.
\end{Lem}
\begin{proof}
Obvious.
\end{proof}

\subsection{}
Denote by $\Ch^{\circ}(\tilde{\Pi},G)\subset\Ch(\tilde{\Pi},G)$ the open subvariety of strongly irreducible representations. (The locus of irreducible representations is open, since it coincides with the stable locus. In this open subset, the stabiliser of each point is a finite extension of the centre of $G$. The locus of strongly irreducibles consists of those points with the smallest possible stabiliser, i.e. the centre of $G$, and so is open since our base field has characteristic 0.\footnote{As is pointed out by a referee, in positive characteristics, there are examples where the cardinality of the stabiliser group is not upper-semicontinuous.}) Let $\tilde{\rho}:\tilde{\Pi}\rightarrow G$ be such that $[\tilde{\rho}]\in\Ch^{\circ}(\tilde{\Pi},G)^{\Gamma}$ and choose some cochain $h_{\ast}$ such that $(\tilde{\rho},h_{\ast})$ is a pre-$(\Gamma,\psi)$-invariant pair. For such a pair, we define a cochain $(k_{\sigma\tau})_{\sigma,\tau\in\Gamma}\in C^2(\Gamma,Z_G)$ by
\begin{equation}
k_{\sigma\tau}:=h_{\sigma\tau}^{-1}\psi_{\sigma}(h_{\tau})h_{\sigma}g_{\sigma\tau}.
\end{equation}
This is well-defined by Lemma \ref{extofinv} (i), since $\Stab_G(\tilde{\rho})=Z_G$.
\begin{Rem}
If we were to work  in the setting of \S \ref{RemFunGrps}, we should define $k_{\sigma\tau}=h_{\sigma\tau}^{-1}\psi_{\tau}^{-1}(h_{\sigma})h_{\tau}g_{\sigma\tau}$ instead. We should also use the right action of $\Gamma$ on $Z_G$ with $\sigma\in\Gamma$ acting by $\psi_{\sigma^{-1}}$, and the differential map on cochains should be changed accordingly.
\end{Rem}

\begin{Prop}\label{defn-c(rho)}
Let $(k_{\sigma\tau})$ be the cochain associated to the given $\tilde{\rho}$ defined as above. Then,
\begin{itemize}
\item[(i)] The cochain $(k_{\sigma\tau})$ is a cocycle; 
\item[(ii)] The cohomology class of $(k_{\sigma\tau})$ does not depend on the choice of $h_{\ast}$;
\item[(iii)] The cochain $h_{\ast}$ can be modified (by $Z_G$ componentwise) to satisfy (\ref{cocycle-rho}) if and only if $(k_{\sigma\tau})$ is a coboundary.
\end{itemize}
\end{Prop}
We will denote by $\mathfrak{c}_{\psi}(\tilde{\rho})$ the cohomology class of $(k_{\sigma\tau})$, which only depends on $\psi$ and $\tilde{\rho}$. The cohomology group $H^2(\Gamma,Z_G)$ only depends on $\Psi$.
\begin{proof}
That $(k_{\sigma\tau})$ is a cocycle can be proved by direct computation. Since this is a known fact (\cite{Wu} and \cite{Sch}), we omit the proof. If $(h_{\sigma})$ is replaced by $(h_{\sigma}x_{\sigma})_{\sigma\in\Gamma}$ for a family $\{x_{\sigma}\in Z_G\}_{\sigma\in\Gamma}$, then $k_{\sigma\tau}$ is multiplied by $x_{\sigma\tau}^{-1}\psi_{\sigma}(x_{\tau})x_{\sigma}$ which is exactly $(\ud x_{\ast})_{\sigma\tau}$. Therefore $[k_{\sigma\tau}]$ does not depend on $h_{\ast}$. We also deduce from this that $(h_{\sigma})$ can be modified to satisfy the desired equality if and only if $(k_{\sigma\tau})$ is a coboundary.
\end{proof}

\subsection{}
We have obtained a partition of $\Ch^{\circ}(\tilde{\Pi},G)^{\Gamma}$ into subsets indexed by $H^2(\Gamma,Z_G)$, and the subset corresponding to the trivial cohomology class consists of the underlying $G$-representations of some admissible $G\rtimes_{\psi}\Gamma$-representations. The following proposition shows that more subsets in this partition can be obtained as restrictions of admissible $G\rtimes\Gamma$-representations. 

Let $\delta$ be the natural map $H^1(\Gamma,\Inn G)\rightarrow H^2(\Gamma,Z_G)$. Recall that $H^1(\Gamma,\Inn G)$ parametrises the equivalence classes of $\psi'\in\Hom(\Gamma,\Aut G)$ with $\Omega(\psi')=\Psi$.
\begin{Prop}\label{changepsi}
Let $\tilde{\rho}\in\Ch^{\circ}(\tilde{\Pi},G)^{\Gamma}$ with $\mathfrak{c}_{\psi}(\tilde{\rho})$ not necessarily equal to $1$. Then there exists some $\psi'$ with $\Omega(\psi')=\Psi$ such that $\mathfrak{c}_{\psi'}(\tilde{\rho})=1$ if and only if $\mathfrak{c}_{\psi}(\tilde{\rho})^{-1}=\delta(\psi')$ for some $\psi'$.
\end{Prop}
\begin{proof}
Given $\psi':\Gamma\rightarrow\Aut G$, such that $\psi'_{\sigma}=\ad x_{\sigma}\circ\psi_{\sigma}$ with $x_{\sigma}\in G$, we put $h'_{\sigma}=x_{\sigma}h_{\sigma}$ and define a cocycle with respect to $\psi'$ by 
\begin{equation}
k'_{\sigma\tau}:=h_{\sigma\tau}^{'-1}\psi_{\sigma}'(h'_{\tau})h'_{\sigma}g_{\sigma\tau}.
\end{equation}
Since both $\psi'$ and $\psi$ are group homomorphisms, there exists some $(d_{\sigma\tau})\in C^2(\Gamma,Z_G)$ such that $$d_{\sigma\tau}=x^{-1}_{\sigma\tau}x_{\sigma}\psi_{\sigma}(x_{\tau}).$$

Now we compute
\begin{equation}
\begin{split}
k'_{\sigma\tau}&=h_{\sigma\tau}^{'-1}\psi_{\sigma}'(h'_{\tau})h'_{\sigma}g_{\sigma\tau}\\
&=h^{-1}_{\sigma\tau}x^{-1}_{\sigma\tau}x_{\sigma}\psi_{\sigma}(x_{\tau}h_{\tau})x_{\sigma}^{-1}x_{\sigma}h_{\sigma}g_{\sigma\tau}\\
&=d_{\sigma\tau}k_{\sigma\tau}.
\end{split}
\end{equation}
We see that $k'_{\sigma\tau}$ is a coboundary if and only if the cohomology class of $(k_{\sigma\tau})$ and that of $(d_{\sigma\tau}^{-1})$ are equal. By definition, $\delta(\psi')$ is the cohomology class of $(d_{\sigma\tau})$.
\end{proof}
We have completed the proof of Theorem \ref{B}.

\section{Generic Conjugacy Classes}\label{GENCONJ}
Now we restrict ourselves to the case of Riemann surfaces. We introduce punctures on Riemann surfaces and study the local monodromy in twisted conjugacy classes. The goal is to give a suitable definition of generic conjugacy classes for $G\rtimes\Gamma~$-character varieties that guarantees irreducibility of $G\rtimes\Gamma$-representations. The definition will be a natural generalisation of the tame case of \cite[Corollary 9.7, Corollary 9.8]{B}.
\subsection{Notations}\label{mononota}
Let $p':\tilde{X}'\rightarrow X'$ be a possibly branched Galois covering of compact Riemann surfaces with $\Aut(\tilde{X}'/X')\cong\Gamma$. Denote by $h$ the genus of $X'$ and $g$ the genus of $\tilde{X}'$. Let $\mathcal{R}\subset X'$ be a finite set of points such that $p'$ is unbranched over $X:=X'\setminus\mathcal{R}$.  Let $I$ be the index set of the elements of $\mathcal{R}$ so that each point of $\mathcal{R}$ is written as $x_j$, $j\in I$. Denote by $\tilde{\mathcal{R}}\subset\tilde{X}'$ the inverse image of $\mathcal{R}$ and write $\tilde{X}:=\tilde{X}'\setminus\tilde{\mathcal{R}}$. Denote by $p$ the restriction of $p'$ to $\tilde{X}$. We fix the base points $\tilde{x}\in\tilde{X}$ and $x=p(\tilde{x})\in X$. For each $x_j\in\mathcal{R}$ and some $\tilde{x}_j\in p^{\prime-1}(x_j)$, put $n_j=|\Stab_{\Gamma}(\tilde{x}_j)|$ with $\Stab_{\Gamma}(\tilde{x}_j)=\lb\sigma_j\rb$, so $\sigma_j\in\Gamma$ is of order $n_j$. (It is well-known that the stabiliser group is necessarily cyclic. See for example \cite[\S 2]{Moo}.) The number $n_j$ only depends on $x_j$. Thus $p'$ is of the form $z\mapsto z^{n_j}$ around each point of $p^{\prime-1}(x_j)$.

For each $j\in I$, let $V_j$ be a sufficiently small neighbourhood of $x_j$. For each $j$, choose a point $y_j\in V_j$ and a loop $l_j$ around $x_j$ based at $y_j$, all with the same orientation. We can choose paths $\lambda_j$ from $x$ to $y_j$ such that  
$\gamma_j:=\lambda_j^{-1}l_j\lambda_j$, together with the generators $\alpha_i$, $\beta_i$, $1\le i\le h$, associated to the genus of $X$, generate $\pi_1(X)$ and satisfy the relation 
\begin{equation}
\prod_{i=1}^h[\alpha_i,\beta_i]\prod_{j\in I}\gamma_j=1.
\end{equation}

The choice of the path $\lambda_j$ determines a point $\tilde{y}_j$ over $y_j$ and thus the connected component of $p^{'-1}(V_j)$ containing $\tilde{y}_j$. Let us denote this connected component by $U_j$ and denote by $\tilde{x}_j$ the point in $U_j$ over $x_j$.  All other connected components of $p^{'-1}(V_j)$ are of the form $\tau(U_j)$ for some $\tau\in\Gamma$. For each such component, we fix such a $\tau$ and thus a point $\tau(\tilde{y}_j)$ in it. The objects associated to these connected components will be indicated by a subscript $(j,\tau)$, for example $U_{j,\tau}=\tau(U_j)$, $\tilde{y}_{j,\tau}=\tau(\tilde{y}_j)$, and in particular, $U_{j,1}=U_j$. If $r_j$ is the lift of $l_j$ starting from $\tilde{y}_j$, then 
\begin{equation}
\tilde{l}_j:=(\sigma_j^{n_j-1}\cdot r_j)\cdots(\sigma_j\cdot r_j)r_j
\end{equation}
is a loop in $U_j$ based at $\tilde{y}_j$, where $\sigma_j\cdot r_j$ is the image of $r_j$ under $\sigma_j$. Let $\tilde{l}_{j,\tau}=\tau(\tilde{l}_{j})$. Again, we can choose paths $\tilde{\lambda}_{j,\tau}$ for $\tilde{x}$ to $\tilde{y}_{j,\tau}$ such that $\tilde{\gamma}_{j,\tau}:=\tilde{\lambda}_{j,\tau}^{-1}\tilde{l}_{j,\tau}\tilde{\lambda}_{j,\tau}$ together with $\tilde{\alpha}_i$, $\tilde{\beta}_i$, $1\le i\le g$,  associated to the genus of $\tilde{X}$, generate $\pi_1(\tilde{X})$ and satisfy a similar relation as for $\gamma_j$'s. Note that the $\tilde{\lambda}_{j,1}$'s are not necessarily the lifts of the $\lambda_j$'s. 

\subsection{Monodromy of $\rho$}\label{monorho}
Let $\rho:\pi_1(X)\rightarrow G\rtimes\Gamma^{op}$ be as in \ref{rho-CD}. Fix $\psi:\Gamma\rightarrow\Aut G$ and write $\bar{G}=G\rtimes_{\psi^{op}}\Gamma^{op}$ and denote by $\tilde{\rho}$ the underlying $G$-representation. Since the end point of the lift of $\gamma_j$ is $\sigma_j(\tilde{x})$, $\gamma_j$ belongs to the $\pi_1(\tilde{X})$-coset in $\pi_1(X)$ corresponding to $\sigma_j$, and so $\rho(\gamma_j)$ lies in the connected component $Gs_{\sigma_j}$. Let $C_j$ be the $G$-conjugacy class of $\rho(\gamma_j)$. A different choice of $\lambda_j$ results in a conjugation of $\gamma_j$ in $\pi_1(X)$, whence a conjugation by $\bar{G}$ of $C_j$. Therefore, unlike usual character varieties, the classes $C_j$ depend on the choices of $\lambda_j$. In practice, we always fix the $\lambda_j$ throughout.

Now we consider what happens on $\tilde{X}$. The lift of $\gamma_j^{n_j}$ is conjugate  to $\tilde{\gamma}_{j,1}$, therefore $\tilde{\rho}(\tilde{\gamma}_{j,1})$ must lie in the conjugacy class $\smash{\tilde{C}_j:=C_j^{n_j}:=\{g^n\mid g\in C_j\}\subset G}$. Then we take $\tau$ to be an element that takes $U_j$ to some $U_{j,\tau}$ that does not meet $U_j$. The lift of $\gamma_{\tau}^{-1}\gamma^{n_j}_j\gamma_{\tau}$ ($\gamma_{\tau}$ is given by the fixed section $\Gamma\rightarrow\pi_1(X)$) is conjugate to $\tilde{\lambda}_{j,\tau}^{-1}\tilde{l}_{j,\tau}\tilde{\lambda}_{j,\tau}=\tilde{\gamma}_{j,\tau}$ in $\pi_1(\tilde{X})$, therefore 
\begin{equation}\label{gammajtau}
\tilde{\rho}(\tilde{\gamma}_{j,\tau})\text{ is conjugate to }\rho(\gamma_{\tau})^{-1}\rho(\gamma_j)^{n_j}\rho(\gamma_{\tau})\text{ by}\Ima\tilde{\rho}
\end{equation}
which lies in the conjugacy class $\rho(\gamma_j)^{-1}\tilde{C}_j\rho(\gamma_j)=s_{\tau}^{-1}\tilde{C}_js_{\tau}=\psi_{\tau}(\tilde{C}_j)=:\tilde{C}_{j,\tau}$. Similarly, $\rho(\gamma_{\tau}^{-1}\gamma_j\gamma_{\tau})\in s_{\tau}^{-1}(C_j)s_{\tau}=:C_{j,\tau}$ which lies in the connected component $Gs_{\tau\sigma_j\tau^{-1}}$. We have $C_{j,\tau}^{n_j}=\tilde{C}_{j,\tau}$. 

\subsection{}
In what follows. we need the following assumption:
\begin{description}
\item[(CYC)\label{CYC}] The image of $$\Omega(\psi):\Gamma\stackrel{\psi}{\rightarrow}\Aut G\rightarrow \Out G$$ is a cyclic subgroup.
\end{description}
This technical assumption will be used in a dimension estimate. It is typically satisfied in one of the following situations:
\begin{itemize}
\item[(i)]
$G$ is $\GL_n$ or almost-simple with root system not of type $D_4$;
\item[(ii)]
$\Gamma$ is cyclic.
\end{itemize}
\subsection{}\label{mathbfssigma}
Recall that $s:\Gamma^{op}\rightarrow G\rtimes\Gamma^{op}$ is the natural section. Fix a maximal torus $T$ and a Borel subgroup $B\subset G$ containing $T$. In each connected component $Gs_{\sigma}$, $\sigma\in\Gamma$, we choose a quasi-central element $\mathbf{s}_{\sigma}\in N_{\bar{G}}(T,B)$, so that $N_{\bar{G}}(T,B)=\sqcup_{\sigma\in\Gamma}T\mathbf{s}_{\sigma}$. Such $\mathbf{s}_{\sigma}$ exists by \cite[Proposition 1.16]{DM94}. By proposition \ref{DM18Prop1.16}, the semi-simple $G$-conjugacy classes in $Gs_{\sigma}$ are parametrised by the $W^{\mathbf{s}_{\sigma}}$-orbits in $$\tilde{\mathbf{T}}_{\sigma}:=T/[T,\mathbf{s}_{\sigma}]\cong (T^{\mathbf{s}_{\sigma}})^{\circ}/(T^{\mathbf{s}_{\sigma}})^{\circ}\cap[T,\mathbf{s}_{\sigma}].$$ with $t\in (T^{\mathbf{s}_{\sigma}})^{\circ}$ representing the class of $t\mathbf{s}_{\sigma}$. Denote by $\mathbf{T}_{\sigma}$ the quotient $\tilde{\mathbf{T}}_{\sigma}/W^{\mathbf{s}_{\sigma}}$. When $\sigma=\sigma_j$ for some $j\in I$, we will write $\mathbf{s}_j$, $\tilde{\mathbf{T}}_{j}$ and $\mathbf{T}_j$ instead of $\mathbf{s}_{\sigma_j}$, $\tilde{\mathbf{T}}_{\sigma_j}$ and $\mathbf{T}_{\sigma_j}$. 

Let $\mathcal{C}=(C_j)_{j\in I}$ be a tuple of $G$-conjugacy classes of $\bar{G}$, with $C_j$ contained in $Gs_{\sigma_j}$. Denote by $\Rep_{\Gamma,\mathcal{C}}(X,G)\subset\Rep_{\Gamma}(X,G)$ the subvariety consisting of $\rho$ satisfying $\rho(\gamma_j)\in C_j$, for all $j\in I$. Semi-simple tuples $\mathcal{C}$ are parametrised by $\prod_j\mathbf{T}_j$. We will define a non-empty subset $\mathbb{T}^{\circ}\subset\prod_j\mathbf{T}_j$ so that $\rho\in\Rep_{\Gamma}(X,G)$ is irreducible whenever $\rho$ lies in $\Rep_{\Gamma,\mathcal{C}}(X,G)$, for some tuple $\mathcal{C}$ whose semi-simple parts correspond to a point of $\mathbb{T}^{\circ}$. The quotient of $\Rep_{\Gamma,\mathcal{C}}(X,G)$ by $G$ will be denoted by $\Ch_{\Gamma,\mathcal{C}}(X,G)$.

\subsection{}
Let $P\subset G$ be a parabolic subgroup containing $B$ and let $L$ be the unique Levi factor of $P$ containing $T$ (\cite[Corollary 8.4.4]{Spr}). In this case we will simply say that $(L,P)$ contains $(T,B)$. There are only finitely many such pairs $(L,P)$. Suppose that $N_{\bar{G}}(L,P)$ meets all connected components of $\bar{G}$. This implies that for any $\sigma\in\Gamma$, the $G$-conjugacy class of $(L\subset P)$ is stable under $\mathbf{s}_{\sigma}$. But $(\mathbf{s}_{\sigma}(L),\mathbf{s}_{\sigma}(P))$ also contains $(T,B)$, so it is necessary that $(\mathbf{s}_{\sigma}(L),\mathbf{s}_{\sigma}(P))=(L,P)$ for any $\sigma\in\Gamma$. Therefore $N_{\bar{G}}(L,P)=\sqcup_{\sigma\in\Gamma}L\mathbf{s}_{\sigma}$. Write $\bar{L}=N_{\bar{G}}(L,P)$.

\begin{Lem}\label{dimest}
Let $L$ and $P$ be as above. Denote by $Z_{\bar{L}}$ and $Z_{\bar{G}}$ the centres of $\bar{L}$ and $\bar{G}$ respectively. If $L\ne G$, then $\dim Z_{\bar{L}}>\dim Z_{\bar{G}}.$
\end{Lem}
\begin{proof}
Because of \ref{CYC}, there is some $\sigma_0\in\Gamma$ such that $\ad\mathbf{s}_{\sigma_0}$ generates the image of $\Omega(\psi)$. Let $\sigma\in\Gamma$, then $\ad\mathbf{s}_{\sigma}$ and $\ad\mathbf{s}_{\sigma_0}^r$ lie in the same connected component of $\Aut G$, for some integer $r$. Since $\ad\mathbf{s}_{\sigma}$ and $\ad\mathbf{s}_{\sigma_0}^r$ both preserve $T$ and $B$, they only differ by $\ad t$ for some $t\in T$. We deduce that they induce the same action on $T$.
Therefore $Z_{\bar{L}}=C_{Z_L}(\mathbf{s}_{\sigma_0})$ and $Z_{\bar{G}}=C_{Z_G}(\mathbf{s}_{\sigma_0})$, since $Z_L$ and $Z_G$ are contained in $T$.

Now $L':=C_L(\mathbf{s}_{\sigma_0})^{\circ}$ is a Levi subgroup of $G':=C_G(\mathbf{s}_{\sigma_0})^{\circ}$. By \cite[Corollaire 1.25]{DM94}, if $L\ne G$, then $L'\ne G'$. By \cite[Proposition 1.23]{DM94}, $Z^{\circ}_{L'}=C_{Z^{\circ}_L}(\mathbf{s}_{\sigma_0})^{\circ}$ and $Z^{\circ}_{G'}=C_{Z^{\circ}_G}(\mathbf{s}_{\sigma_0})^{\circ}$. Then the lemma follows from the result for Levi subgroups of connected groups.
\end{proof}
\begin{Rem}
If $\mathbf{s}_{\sigma_{1}}$ and $\mathbf{s}_{\sigma_{2}}$ are unrelated, then we do not have a good control over the dimension of $C_{Z_L}(\mathbf{s}_{\sigma_1})\cap C_{Z_L}(\mathbf{s}_{\sigma_2})$. This is the reason why we impose \ref{CYC}.
\end{Rem}

\subsection{}
For any connected reductive algebraic group $H$, denote by 
$$\mathbf{D}_H:H\longrightarrow Z_H^{\circ}/(Z_H^{\circ}\cap[H,H])$$
the natural projection, identifying $H/[H,H]\cong Z_H^{\circ}/(Z_H^{\circ}\cap[H,H])$. For $H=\GL_n$, it can be identified with the determinant map. Note that $Z_H^{\circ}\cap[H,H]$ is a finite group.

Each element $\mathbf{t}_{\sigma}\in\mathbf{T}_{\sigma}$ is a $W^{\mathbf{s}_{\sigma}}$-orbit. Each element $t_{\sigma}\in\mathbf{t}_{\sigma}$ is a coset in $(T^{\mathbf{s}_{\sigma}})^{\circ}$. We fix a representative in $(T^{\mathbf{s}_{\sigma}})^{\circ}$ of each such $t_{\sigma}$, also denoted by $t_{\sigma}$. The choice of such representative will not matter. If $\sigma=\sigma_j$ for some $j\in I$, then we write $t_j$ and $\mathbf{t}_j$ instead of $t_{\sigma_j}$ and $\mathbf{t}_{\sigma_j}$.
\begin{Defn}\label{GCC}
A tuple of semi-simple conjugacy classes parametrised by $(\mathbf{t}_j)_{j\in I}$ is generic if the following condition is satisfied. For
\begin{itemize}
\item[-]
any$(L,P)$ containing $(T,B)$ with $P\ne G$ such that $N_{\bar{G}}(L,P)$ meets all connected components of $\bar{G}$, and 
\item[-]
any tuple $(t_j)_{j\in I}$ with $t_j\in\mathbf{t}_j$, 
\end{itemize}
the element
\begin{equation}\label{eq-GCC}
\prod_{j\in I}\prod_{\tau\in\Gamma/\lb\sigma_j\rb}\mathbf{D}_L(\mathbf{s}_{\tau}(t_j^{n_j}\mathbf{s}_j^{n_j}))\in Z_L^{\circ}/(Z_L^{\circ}\cap[L,L])
\end{equation}
is not equal to the identity, where $\mathbf{s}_{\tau}$ acts on $T$ by conjugation and $\tau$ runs over a set of representatives of $\Gamma/\lb\sigma_j\rb$. A tuple of conjugacy classes $\mathcal{C}$ is generic if the tuple of the conjugacy classes of the semi-simple parts of $\mathcal{C}$ is generic.
\end{Defn}

One can verify that $\mathbf{D}_L(\mathbf{s}_{\tau}(t^{n_j}\mathbf{s}_j^{n_j}))$ has constant value for $t\in[T,\mathbf{s}_j]$ so it only depends on the coset $t_j$. 

\begin{Rem}\label{^nsurjects}
Since $\mathbf{s}_j^{n_j}\in T$, the morphism of varieties $(T^{\mathbf{s}_{j}})^{\circ}\mathbf{s}_j\rightarrow T$, $t\mathbf{s}_j\rightarrow (t\mathbf{s}_j)^{n_j}$, surjects onto a connected component of $T^{\mathbf{s}_j}$.
\end{Rem}

\begin{Lem}\label{dimextimate}
Put $$(Z^{\circ}_{L})^{\Gamma}=\{z\in Z^{\circ}_L\mid \mathbf{s}_{\tau}(z)=z,\text{ for all $\tau\in\Gamma$}\},$$and similarly for $\left(Z_L^{\circ}/(Z_L^{\circ}\cap[L,L])\right)^{\Gamma}$. Then we have $$\dim\left(Z_L^{\circ}/(Z_L^{\circ}\cap[L,L])\right)^{\Gamma}=\dim (Z^{\circ}_L)^{\Gamma}=\dim Z_{\bar{L}}.$$
\end{Lem}
\begin{proof}
The first equality is obvious. We have $Z_{\bar{L}}=Z_L^{\Gamma}$ whence the second equality.
\end{proof}
\begin{Lem}\label{surjZGamma}
For any $j\in I$, the map $$\prod_{\tau\in\Gamma/\lb\sigma_j\rb}\mathbf{D}_L\circ\mathbf{s}_{\tau}:T^{\mathbf{s}_j}\longrightarrow \left(Z_L^{\circ}/(Z_L^{\circ}\cap[L,L])\right)^{\Gamma}$$ is a group homomorphism under which each connected component of $T^{\mathbf{s}_j}$ is mapped surjectively onto a connected component of the target.
\end{Lem}
Note that the image of $\mathbf{D}_L$ lies in $Z_L^{\circ}/(Z_L^{\circ}\cap[L,L])$. The point of this lemma is to show that after taking the "average" over $\Gamma/\lb\sigma_j\rb$, the image is $\Gamma$-invariant.
\begin{proof}
Let $t\in T^{\mathbf{s}_j}$ and $\sigma\in\Gamma$. We have $$\mathbf{s}_{\sigma}\big(\prod_{\tau\in\Gamma/\lb\sigma_j\rb}\mathbf{D}_L(\mathbf{s}_{\tau}(t))\big)=\prod_{\tau\in\Gamma/\lb\sigma_j\rb}\mathbf{D}_L(\mathbf{s}_{\sigma}\mathbf{s}_{\tau}(t)).$$ Note that $\mathbf{s}_{\sigma}\mathbf{s}_{\tau}$ differs from $\mathbf{s}_{\sigma\tau}$ by an element of $T$, and therefore they have the same action on $T$. Also, all elements in a coset $\tau\lb\sigma_j\rb$ have the same action on $T^{\mathbf{s}_j}$. The right hand side of the equality is thus equal to $\prod_{\tau\in\Gamma/\lb\sigma_j\rb}\mathbf{D}_L(\mathbf{s}_{\tau}(t))$. So the image is $\Gamma$-invariant.

Let $z\in (Z^{\circ}_{L})^{\Gamma}$. Then $$\prod_{\tau\in\Gamma/\lb\sigma_j\rb}\mathbf{D}_L(\mathbf{s}_{\tau}(z))=z^{|\Gamma|/n_j}.$$ The map $z\mapsto z^{|\Gamma|/n_j}$ is a surjection onto $((Z^{\circ}_{L})^{\Gamma})^{\circ}$, whence the surjectivity for other components.
\end{proof}

\subsection{}
Denote by $\tilde{\mathbb{T}}\subset\prod_j\tilde{\mathbf{T}}_j$ the closed subvariety defined by
\begin{equation}
\prod_{j\in I}\prod_{\tau\in\Gamma/\lb\sigma_j\rb}\mathbf{D}_G(\mathbf{s}_{\tau}(t_j^{n_j}\mathbf{s}_j^{n_j}))=1,
\end{equation}
with $t_i\in\tilde{\mathbf{T}}_j$. Write $\mathbf{W}:=\prod_jW^{\mathbf{s}_j}$. Then $\prod_j\mathbf{T}_j=(\prod_j\tilde{\mathbf{T}}_j)/\mathbf{W}$ and the action of $\mathbf{W}$ preserves $\tilde{\mathbb{T}}$. Define $\mathbb{T}:=\tilde{\mathbb{T}}/\mathbf{W}\subset\prod_j\mathbf{T}_j$.
\begin{Prop}
The subset of generic semi-simple conjugacy classes $\mathbb{T}^{\circ}\subset\mathbb{T}$ is Zariski open and non-empty.
\end{Prop}
\begin{proof}
Let $Z\subset \prod_j\tilde{\mathbf{T}}_j$ be the closed subset defined by: $(t_j)$ belongs to $Z$ if for some $(L,P)\supset(T,B)$ with $P\ne G$ such that $N_{\bar{G}}(L,P)$ meets all connected components of $\bar{G}$,
\begin{equation}
\prod_{j\in I}\prod_{\tau\in\Gamma/\lb\sigma_j\rb}\mathbf{D}_L(\mathbf{s}_{\tau}(t_j^{n_j}\mathbf{s}_j^{n_j}))= 1.
\end{equation}
By Remark \ref{^nsurjects} and Lemma \ref{surjZGamma}, the image of the map: $$\prod_{j\in I}\prod_{\tau\in\Gamma/\lb\sigma_j\rb}\mathbf{D}_L\circ\mathbf{s}_{\tau}(-)^{n_j}:\prod_j\tilde{\mathbf{T}}_j\longrightarrow\left(Z_L^{\circ}/(Z_L^{\circ}\cap[L,L])\right)^{\Gamma}$$is a connected component.
By Lemma \ref{dimest} and Lemma \ref{dimextimate}, we have $\dim Z<\dim\tilde{\mathbb{T}}$.

The finite group $\mathbf{W}$ acts on $\prod_j\tilde{\mathbf{T}}_j$ and preserves the closed subsets $\tilde{\mathbb{T}}$ and $\cup_{w\in\mathbf{W}}w\cdot Z$. Define $$\tilde{\mathbb{T}}^{\circ}=\tilde{\mathbb{T}}\setminus \bigcup_{w\in\mathbf{W}}w\cdot Z.$$ It is a $\mathbf{W}$-invariant open subset of $\tilde{\mathbb{T}}$, and is non-empty for dimension reason. Then by definition $(\mathbf{t}_j)\in\mathbb{T}^{\circ}$ if and only if all of its fibres in $\tilde{\mathbb{T}}$ are contained in $\tilde{\mathbb{T}}^{\circ}$ and so $\mathbb{T}^{\circ}=\tilde{\mathbb{T}}^{\circ}/\mathbf{W}$ is open and non-empty.
\end{proof}

\subsection{}
We conclude this section by the following proposition.
\begin{Prop}\label{Gen->Irr}
Suppose that $\mathcal{C}$ is a tuple of generic conjugacy classes. Then every element of $\Rep_{\Gamma,\mathcal{C}}(X,G)$ is an irreducible $\bar{G}$-representation.
\end{Prop}
\begin{proof}
Fix $T\subset B$ and $\mathbf{s}_{\sigma}$ as in \S \ref{mathbfssigma}. Suppose that $\rho\in\Rep_{\Gamma,\mathcal{C}}(X,G)$ is not irreducible. Then there exists some proper parabolic subgroup $P\subset G$ such that $N_{\bar{G}}(P)$ meets all connected components of $\bar{G}$ and $\Ima\rho\subset N_{\bar{G}}(P)$. Up to a $G$-conjugation we can assume that $P$ contains $B$. Let $L$ be the unique Levi factor of $P$ containing $T$. 

Put $c_j:=\rho(\gamma_j)\in P\mathbf{s}_j$, then $\smash{c_j^{n_j}\in P}$. For $\tau$ representing a coset in $\Gamma/\lb\sigma_j\rb$, $\tilde{\rho}(\tilde{\gamma}_{j,\tau})$ is $P$-conjugate to $\smash{\mathbf{s}_{\tau}(c_j^{n_j})}$ by (\ref{gammajtau}). Let $\pi_L:P\rightarrow L$ be the natural projection. Using a presentation of $\pi_1(\tilde{X})$ by $\{\tilde{\gamma}_{j,\tau},~\tilde{\alpha}_i,~\tilde{\beta}_i\}$, we find
\begin{equation}
\prod_{j\in I}\prod_{\tau\in\Gamma/\lb\sigma_j\rb}\mathbf{D}_L\circ\pi_L(\mathbf{s}_{\tau}(c_j^{n_j}))=1.
\end{equation}
Note that the value of $\mathbf{D}_L$ only depends on the semi-simple parts. The semi-simple part $c_{j,s}$ of $c_j$ is contained in $G\mathbf{s}_j$ because all unipotent elements are contained in $G$. In particular, $c_{j,s}\in P\mathbf{s}_j$. It is therefore $P$-conjugate to an element of $N_{\bar{G}}(T,B)$ and is further $L$-conjugate to an element of $(T^{\mathbf{s}_{j}})^{\circ}\mathbf{s}_j$ by Proposition \ref{DM18Prop1.16}. Now, the above relation contradicts the definition of generic conjugacy classes.
\end{proof}

\section{Double Coverings}\label{DOUCOV}

We apply our previous results to a simple but important example: $\GL_n\rtimes\lb\sigma\rb~$-character varieties. The geometric setting of this section is as follows. We have a double covering of compact Riemann surfaces $p':\tilde{X}'\rightarrow X'$. It has Galois group $\Gamma\cong\mathbb{Z}/2\mathbb{Z}$. By Hurwitz formula, there can only be an even number of ramification points. Let $\mathcal{R}\subset X'$, $\tilde{\mathcal{R}}$ and $p:\tilde{X}\rightarrow X$ be as in \S \ref{mononota}. The induced map between fundamental groups will also be denoted by $p:\tilde{\Pi}\rightarrow\Pi$. To simplify the discussion, we will make the following assumption on $\mathcal{R}$: If $p'$ is branched, then $\mathcal{R}$ is exactly the set of ramification points, and we write $|\mathcal{R}|=2k$; Otherwise $\mathcal{R}$ is a non empty finite set, and we write $|\mathcal{R}|=k$.

\subsection{Automorphisms of $\GL_n(\kk)$}\label{autoGLn}
We will write $G=\GL_n(\kk)$. Let $J_0$ be the skew-diagonal matrix defined by $(J_0)_{ij}=\delta_{i,n+1-j}$, and if $n$ is even, put $$t=\diag(1,\ldots,1,-1,\ldots,-1)$$ with an equal number of $1$ and $-1$. Put $J=tJ_0$. We define $\sigma_o$ as the automorphism $g\mapsto J_0g^{-t}J^{-1}_0$. If $n$ is even, we also consider the automorphism $\sigma_s$ which sends $g$ to $Jg^{-t}J^{-1}$. By \cite[Lemma 2.9]{LiSe}, there exists a unique conjugacy class of outer automorphisms of order 2 of $G$ if $n$ is odd, and there are two such conjugacy classes if $n$ is even. These two conjugacy classes have $\sigma_o$ and $\sigma_s$ as representatives respectively. The centralisers of $\sigma_o$ and $\sigma_s$ are orthogonal groups and symplectic groups respectively. 

A choice of an involution $\sigma\in\Aut G$ defines a semi-direct product $G\rtimes \mathbb{Z}/2\mathbb{Z}$, which we will denote by $G\rtimes\lb\sigma\rb$. \begin{Prop}\label{classify-GZ/2Z}
The semi-direct products $G\rtimes\lb\sigma_o\rb$, $G\rtimes\lb\sigma_s\rb$ and $G\times\mathbb{Z}/2\mathbb{Z}$ are not isomorphic, unless $n$ is odd, in which case $G\rtimes\lb\sigma_o\rb$ is isomorphic to $G\rtimes\lb\sigma_s\rb$. Any semi-direct product $G\rtimes\lb\sigma\rb$ defined by an involution is isomorphic to one of these groups.
\end{Prop}
\begin{proof}
We apply Theorem \ref{para-GGamma} to this special case. A homomorphism $\Psi:\mathbb{Z}/2\mathbb{Z}\rightarrow\Out G$ is determined by the image of $1\in\mathbb{Z}/2\mathbb{Z}$. There are two cases: $\Psi(1)=1$ or $\Psi(1)\ne 1$. In the first case, $H^2(\mathbb{Z}/2\mathbb{Z},\kk^{\ast})$ is trivial and the resulting semi-direct product is the direct product. In the second case, $H^2(\mathbb{Z}/2\mathbb{Z},\kk^{\ast})=\{\pm1\}$, and we need to distinguish the case of even $n$ and odd $n$. It is easy to see that the set of conjugacy classes of involutions in $\Aut G\setminus\Inn G$ is in bijection with $H^1(\mathbb{Z}/2\mathbb{Z},\PGL_n)$. If $n$ is odd, then $H^1(\mathbb{Z}/2\mathbb{Z},\PGL_n)$ consists of one point, mapped to $+1$ by $\delta$. If $n$ is even, then $H^1(\mathbb{Z}/2\mathbb{Z},\PGL_n)$ consists of two points, mapped surjectively onto $H^2(\mathbb{Z}/2\mathbb{Z},\kk^{\ast})$. Therefore, there are two isomorphism classes if $n$ is even and one isomorphism class if $n$ is odd.
\end{proof}

If $n$ is even, we will write $\leftidx{^s\!}{\bar{G}}=G\rtimes\lb\sigma_s\rb$ and $\leftidx{^o\!}{\bar{G}}=G\rtimes\lb\sigma_o\rb$. We write $\bar{G}=\leftidx{^s\!}{\bar{G}}$ or $\leftidx{^o\!}{\bar{G}}$ when there is no need to distinguish them. When $n$ is odd, we will simply write $\bar{G}=\leftidx{^o\!}{\bar{G}}$. Note however, that the above classification also works for $G=\SL_n(\kk)$. But if $G=\PGL_n(\kk)$, the centre is trivial, so $\leftidx{^s\!}{\bar{G}}$ is isomorphic to $\leftidx{^o\!}{\bar{G}}$, and they are actually isomorphic to $\Aut(G)$. 

\begin{Eg}\label{stratify-Ch}
Let $\sigma$ be an outer automorphism. By Theorem \ref{B} and the fact that $H^2(\Gamma,\kk^{\ast})\cong\{\pm 1\}$, the fixed points locus $\Ch^{\circ}(\tilde{\Pi},G)^{\Gamma}$ in the usual character variety of $\tilde{\Pi}$ is divided into two parts. The two subsets corresponding to $+1$ and $-1$ are the underlying representations of admissible $\leftidx{^o\!}{\bar{G}}$-representations and admissible $\leftidx{^s\!}{\bar{G}}$-representations of $\Pi$ respectively.
\end{Eg}

\begin{Rem}
The $\bar{G}$-character variety can also be regarded as a moduli of local systems. Given a local system $\mathcal{L}$ on $\tilde{X}'$, we say that $\mathcal{L}$ is invariant of signature "$+$" if there is an isomorphism $\Phi:\mathcal{L}\isom\theta^{\ast}\mathcal{L}^{\vee}$, where $\theta^{\ast}\mathcal{L}^{\vee}$ is the pull-back of the dual local system via the non trivial covering transformation $\theta$, satisfying $\theta^{\ast}\Phi^{-t}\circ\Phi=\Id$, where $\Phi^{-t}$ is the transpose-inverse. We say that $\mathcal{L}$ is invariant of signature "$-$" if there is such an isomorphism satisfying $\theta^{\ast}\Phi^{-t}\circ\Phi=-\Id$. The $\leftidx{^o\!}{\bar{G}}$-character variety (resp. $\leftidx{^s\!}{\bar{G}})$ is then the moduli of invariant $(\mathcal{L},\Phi)$ on $\tilde{X}'$ of signature $+$ (resp. $-$). 
\end{Rem}

\subsection{Maximal Parabolic Subgroups of $G\rtimes\lb\sigma\rb$}\label{maxparaGbar}
Let $\sigma=\sigma_s$ or $\sigma_o$. We are only interested in those parabolic subgroups that meet both connected components. 

By \S \ref{PmeetsG1}, a parabolic subgroup $P\subset\bar{G}$ of the form $N_{\bar{G}}(P^{\circ})$ meets the connected component $G\sigma$ if and only if the $G$-conjugacy class of $P^{\circ}$ corresponds to a $\sigma$-stable subdiagram of the Dynkin diagram of $G^{\circ}$. Therefore if we take for $T$ and $B$ the subgroups of diagonal matrices and of upper triangular matrices, a standard parabolic $P^{\circ}$ containing $B$ such that $N_{\bar{G}}(P^{\circ})$ meets $G\sigma$ and is maximal among the parabolics satisfying these properties must consist of matrices of the form
\begin{equation}\label{max-para}
  \renewcommand{\arraystretch}{1.2}
\left(
  \begin{array}{ c | c c | c }
    A &  \multicolumn{2}{c}{\ast} & \ast  \\
    \cline{1-3}
     0 & & & \multirow{2}*{$\ast$}  \\
     0 & \multicolumn{2}{c|}{\raisebox{.6\normalbaselineskip}[0pt][0pt]{$B$}} &   \\
    \cline{2-4}
     \mc{0} & 0 & 0 & C \\
     \end{array}
  \right),
\end{equation}
where $A$ and $C$ are square matrices of the same size. Normalisers of such $P^{\circ}$'s ($=P^{\circ}\sqcup P^{\circ}\sigma$) are the representatives of the $G^{\circ}$-conjugacy classes of maximal parabolic subgroups of $\bar{G}$ that meet $G\sigma$.

\subsection{Semisimple Conjugacy Classes of $G\rtimes\lb\sigma\rb$}\label{ssclasses}
Let $\sigma=\sigma_s$ or $\sigma_o$. Let $T\subset G$ be the maximal torus of diagonal matrices and let $W=W_G(T)$, which admits an action of $\sigma$ induced from $G$. Denote by $W^{\sigma}$ the subgroup of $\sigma$-fixed points. The Borel subgroup of upper triangular matrices is also stable under $\sigma$.

Denote by $(T^{\sigma})^{\circ}$ the connected centraliser of $\sigma$ in $T$. It consists of matrices of the form
\begin{equation}\label{repres-ss-Gs}
\diag(a_1,\ldots,a_m,a_m^{-1},\ldots,a_1^{-1}),~a_i\in \kk^{\ast}
\end{equation}
if $n=2m$, and with an extra $1$ in the middle if $n$ is odd. 

Denote by $[T,\sigma]$ the commutator. It consists of 
\begin{equation}
\diag(b_1,\ldots,b_m,b_m,\ldots,b_1),~b_i\in \kk^{\ast},
\end{equation}
if $n=2m$, and with an extra entry $b\in\kk^{\ast}$ in the middle if $n$ is odd.

So $S:=(T^{\sigma})^{\circ}\cap[T,\sigma]$ consists of
\begin{equation}
\diag(e_1,\ldots,e_m,e_m,\ldots,e_1),~e_i\in\{\pm1\},
\end{equation}
if $n=2m$, and with an extra $1$ in the middle if $n$ is odd. 

By Proposition \ref{DM18Prop1.16}, the semisimple conjugacy classes in $G\sigma$ are parametrised by the $W^{\sigma}$-orbits on the quotient $(T^{\sigma})^{\circ}/S$. The two automorphisms $\sigma_s$ and $\sigma_o$ have the same action on $T$ and on $W$, so we see that there is no difference between them regarding the parametrisation. We have $W^{\sigma}\cong(\mathbb{Z}/2\mathbb{Z})^m\rtimes\mathfrak{S}_m$, where $\mathfrak{S}_m$ acts by permuting the factors. According to the proposition,  
the following operations leave a representative (\ref{repres-ss-Gs}) in the same conjugacy class:
\begin{itemize}
\item[-] Interchanging $a_i$ and $a_i^{-1}$ (action of $\mathbb{Z}/2\mathbb{Z}$);
\item[-] Changing any pair $(a_i,a_i^{-1})$ to $(-a_i,-a_i^{-1})$ (quotient by $S$);
\item[-] Symmetrically permuting the $a_i$'s and $a_i^{-1}$'s (action of $\mathfrak{S}_m$).
\end{itemize}

\subsection{Generic conjugacy classes in $\GL_n\rtimes\lb\sigma\rb$}\label{GenConGLnsigma}We consider two situations.

(i). Branched covering.
Let $\mathcal{C}=(C_1,\ldots,C_{2k})$ be a $2k$-tuple of semi-simple $G$-conjugacy classes contained in $G\sigma$. Each $C_j$ is determined by an $m$-tuple $(a_{j,1},\ldots,a_{j,m})$ with each $a_{j,i}\in\kk^{\ast}$ (See (\ref{repres-ss-Gs})). Write $\Lambda=\{1,\ldots,m\}$. For any $j$, any subset $\mathbf{A}\subset\Lambda$ and any $|\mathbf{A}|$-tuple of signs $\smash{\mathbf{e}=(e_i)_{i\in\mathbf{A}}}$, $e_i\in\{\pm 1\}$, write $\smash{[\mathbf{A},\mathbf{e}]_j=\prod_{i\in\mathbf{A}}(a_{j,i}^{e_i})^2}$. 

We explain below that the tuple $\mathcal{C}$ is generic if and only if for any $1\le l \le m$, any $2k$-tuple $(\mathbf{A}_1,\ldots,\mathbf{A}_{2k})$ of subsets of $\Lambda$ such that $|\mathbf{A}_1|=\cdots=|\mathbf{A}_{2k}|=l$, and any $2k$-tuple of $l$-tuples $(\mathbf{e}^1,\ldots,\mathbf{e}^{2k})$ of signs, we have
\begin{equation}
[\mathbf{A}_1,\mathbf{e}^1]_1\cdots[\mathbf{A}_{2k},\mathbf{e}^{2k}]_{2k}\ne1.
\end{equation}

Here we unravel Definition \ref{GCC}. Since $p'$ is branched at each $x_j$, we have $n_j=2$ and the set $\Gamma/\lb\sigma_j\rb$ is trivial. Let $t_j\sigma=t_j\mathbf{s}_j$ be a representative of $C_j$ and $t_j$ is of the form (\ref{repres-ss-Gs}). Then $t_j^2\sigma^2=t_j^2$. We then take a maximal $\sigma$-stable standard parabolic subgroup as in \ref{max-para}. The map $\mathbf{D}_L$ consists of three determinant maps: $\det_A$, $\det_B$ and $\det_C$. The map $\det_B$ always vanishes and $\det_C$ gives the inverse of $\det_A$ in view of (\ref{repres-ss-Gs}). Suppose that the matrix $A$ has rank $l$, then the $\det_A$-component of $\mathbf{D}_L(t_j^2)$ is $\prod_{1\le i\le l}a_{j,i}^2$. But we could have used other representatives of $C_j$ that are different from $t_j\sigma$. Consequently, instead of using $1\le i\le l$, we can use $i\in\mathbf{A}$ for any subset $\mathbf{A}\subset\Lambda$ of size $l$. An $l$-tuple of signs $\mathbf{e}$ arises from the action of $(\mathbb{Z}/2\mathbb{Z})^m\subset W^{\sigma}$. We get $\smash{[\mathbf{A},\mathbf{e}]_j=\prod_{i\in\mathbf{A}}(a_{j,i}^{e_i})^2}$ as defined above.
\begin{Rem}
The definition of generic conjugacy classes (Definition \ref{GCC}) involves all pairs $(L,P)$, whereas the above discussion only considers those maximal $P$. If $P$ is not maximal, then $\mathbf{D}_L$ would have more components, say, $\det_{A'}$ and $\det_{C'}$ with $\det_{C'}=\det_{A'}^{-1}$. Then the genericity condition requires that some certain product of the eigenvalues that arises as $\det_{A'}$ is not equal to one. But this gives no new restriction on the eigenvalues.
\end{Rem}

(ii). Unbranched covering. Let $\mathcal{C}=(C_1,\ldots,C_{k})$ be a $k$-tuple of semi-simple conjugacy classes in $G$. For each $j$, let $(a_{j,1},\ldots,a_{j,n})$, $a_{j,i}\in\kk^{\ast}$ be the eigenvalues of $C_j$ . Write $\Lambda=\{1,\ldots,n\}$. For any $j$, any subset $\mathbf{A}\subset\Lambda$, write $\smash{[\mathbf{A}]_j=\prod_{i\in\mathbf{A}}a_{j,i}}$. 

In this case, the tuple $\mathcal{C}$ is generic if and only if for any $1\le l \le m$, any two $k$-tuples $(\mathbf{A}_1,\ldots,\mathbf{A}_{k})$ and $(\mathbf{B}_1,\ldots,\mathbf{B}_{k})$ of subsets of $\Lambda$ such that 
\begin{itemize}
\item
$|\mathbf{A}_1|=\cdots=|\mathbf{A}_{k}|=|\mathbf{B}_1|=\cdots=|\mathbf{B}_{k}|=l$;
\item
$\mathbf{A}_j\cap \mathbf{B}_j=\varnothing$, for all $j$,
\end{itemize}
 we have
\begin{equation}\label{GenConGLn1}
[\mathbf{A}_1]_1\cdots[\mathbf{A}_{k}]_{k}[\mathbf{B}_1]_1^{-1}\cdots[\mathbf{B}_{k}]_{k}^{-1}\ne1.
\end{equation}

Again we unravel Definition \ref{GCC}. Now $p'$ is unbranced at each $x_j$, so we have $n_j=1$ and $\Gamma/\lb\sigma_j\rb=\Gamma$. The representative $t_j$ of $C_j$ as in Definition \ref{GCC} is now a diagonal matrix with eigenvalues $(a_{j,1},\ldots,a_{j,n})$. The map $\mathbf{D}_L$ is the same as in the branched case, consisting of three determinant maps, but we should evaluate at $t_j\sigma(t_j)$ due to the product over $\Gamma$. The determinant map $\det_A$ gives a term of the form $[\mathbf{A}]_j[\mathbf{B}]_j^{-1}$ and $\det_C$ gives its inverse.

\subsection{Irreducible Subgroups of $\GL_n\rtimes\lb\sigma\rb$}
Let $H_0\subset\GL_n$ be a topologically finitely generated closed subgroup, i.e. $H_0=H(\mathbf{x})$ for some finite tuple $\mathbf{x}$ of closed points of $\GL_n$ (see \S \ref{H(x)}) and let $H\subset\GL_n\rtimes\lb\sigma\rb$ be a closed subgroup generated by $H_0$ and an element $x_0\sigma\in\GL_n\sigma$ satisfying
\begin{itemize}
\item[(i)]
$x_0\sigma$ normalises $H_0$;
\item[(ii)] 
$(x_0\sigma)^2\in H_0$.
\end{itemize}
In particular, $H_0=H\cap \GL_n$.
\begin{Prop}\label{irrsbgpGLsigma}
If $H$ is irreducible, then the natural representation $\kk^n$ of $\GL_n$ is a direct sum of pairwise non isomorphic irreducible $H_0$-representations, say $\bigoplus_jV_j$, and the centraliser $C_{\GL_n}(H)$ is a finite abelian group.
\end{Prop}
\begin{proof}
The second statement follows from the proof of the first.

Since $H$ is irreducible, $H_0$ is completely reducible in $\GL_n$ by \cite[Lemma 6.12]{BMR}, and so $\kk^n$ can be written as a direct sum of irreducible $H_0$-representations, say 
\begin{equation}\label{H0irrM}
\kk^n\cong\bigoplus_jV_j^{\oplus r_j},
\end{equation}
where $V_j$ is not isomorphic to $V_{j'}$ whenever $j\ne j'$. We see that
\begin{equation}
C_{\GL_n}(H_0)\cong\prod_j\GL_{r_j},
\end{equation}
where each entry of an element of $\GL_{r_j}$ is identified with a scalar endomorphism of $V_j$.

Let us now prove that $r_j=1$ for all $j$. Note that the centre of $\GL_n\rtimes\lb\sigma\rb$ is $\{\pm \Id\}$, so the irreducibility of $H$ is equivalent to having finite centraliser in $\GL_n$ by Theorem \ref{st=irr}. (The group $H$ is topologically generated by finitely many, say $m$, elements. Then consider the conjugation action of $G^{\circ}$ on $G^m$, $G=\GL_n\rtimes\lb\sigma\rb$.) The proof is achieved by exploiting the fact that $C_{\GL_n}(H)$ has zero dimension. Since $x_0\sigma$ normalises $H_0$, it normalises $C_{\GL_n}(H_0)$. Also, $(x_0\sigma)^2\in H_0$, so $x_0\sigma$ defines an order 2 (in particular, semi-simple) automorphism of $C_{\GL_n}(H_0)$ as an algebraic group. In order for an element of $C_{\GL_n}(H_0)$ to centralise $H$, it suffices for it to commute with $x_0\sigma$. Choose a $x_0\sigma$-stable maximal torus of $C_{\GL_n}(H_0)$ and consider its root system with respect to this maximal torus. (Such maximal torus exists because $x_0\sigma$ induces a semi-simple automorphism.) If the action of $x_0\sigma$ permutes two root subgroups of $C_{\GL_n}(H_0)$, then $C_{\GL_n}(H)$ would contain a root subgroup, and so have positive dimension, which is a contradiction. So $x_0\sigma$ fixes all roots of $C_{\GL_n}(H_0)$. But then it would be an inner semi-simple automorphism of the derived subgroup of $C_{\GL_n}(H_0)$, thus fixes pointwise a maximal torus of it. Again using the zero-dimension of $C_{\GL_n}(H)$, we deduce that the derived subgroup of $C_{\GL_n}(H_0)$ must have rank 0, i.e. $C_{\GL_n}(H_0)$ is a torus. This means that $r_j=1$ for all $j$. As a subgroup, $C_{\GL_n}(H)$ must be abelian.
\end{proof}
\begin{Rem}\label{LieirrsbgpGLsigma}
(i). The proof of the above proposition also works in odd positive characteristics. (ii). Taking the Lie algebras, we get 
$$C_{\mathfrak{gl}_n}(H_0)\cong\bigoplus_j \kk,\quad C_{\mathfrak{gl}_n}(H)=\{0\},$$ unde our assumption that $\kk=\mathbb{C}$.
\end{Rem}
\begin{Rem}
Here is an alternative proof of the assertion that $C_{\GL_n}(H_0)$ is a torus.\footnote{This proof is suggested to the author by a referee.} By \cite[Theorem 7.2]{St}, $x_0\sigma$ normalises a Borel subgroup $B_0$ of $C_{\GL_n}(H_0)$. Let $U_0$ be the unipotent radical of $B_0$. Since $U_0\subset C_{\GL_n}(H_0)$, we have $H_0\subset C_{\GL_n}(U_0)$, and so $H\subset N_G(U_0)$. By \cite[Proposition 3.1]{BT}, there exists a parabolic subgroup $P\subset\GL_n$ such that $N_G(U_0)\subset N_G(P)$ and $U_0\subset R_u(P)$, where $R_u(P)$ is the unipotent radical of $P$. Now we have $H\subset N_G(P)$. If $U_0\ne\{1\}$, then $P\subset\GL_n$ is a proper subgroup, since $\GL_n$ is reductive. But this contradicts the assumption that $H$ is $G$-irreducible. We deduce that $B_0$ is a torus, and so $C_{\GL_n}(H_0)$, being reductive, must be a torus as well.
\end{Rem}

\subsection{Smoothness and Dimension}
Let $g$ be the genus of $X$. In the branched case, we have a $G$-conjugacy class $C_j\subset G\sigma$ for each $1\le j\le 2k$, and the $\GL_n\rtimes\lb\sigma\rb~$-character variety $\Ch_{\Gamma,\mathcal{C}}(X,G)$ is defined by
\begin{equation}
\{(A_i,B_i)_i(Y_j)_j\in \GL_n^{2g}\times \prod^{2k}_j C_j|\prod_{i=1}^g[A_i,B_i]\prod_{j}Y_j=1\}\ds \GL_n.
\end{equation}
In the unbranched case, we have a conjugacy class $C_j\subset G$ for each $1\le j\le k$, and the character variety $\Ch_{\Gamma,\mathcal{C}}(X,G)$ is defined by
\begin{equation}
\{(A_i,B_i)_i(X_j)_j\in \GL_n^{2g}\times \prod_{j=1}^{k} C_j\mid A_1\sigma(B_1)A_1^{-1}B_1^{-1}\prod_{i=2}^g[A_i,B_i]\prod_{j=1}^{k}X_j=1\}\ds\GL_n.
\end{equation}

\begin{Thm}
If the tuple of conjugacy classes $\mathcal{C}=(C_j)_j$ is generic, then the variety $\Ch_{\Gamma,\mathcal{C}}(X,G)$ above is smooth and its dimension is given by $$
\dim\Ch_{\Gamma,\mathcal{C}}(X,G)=(2g-2)\dim\GL_n+\sum_{\text{all }j}\dim C_j.$$
\end{Thm}
\begin{proof}
The proof is completely analogous to \cite[Theorem 2.1.5]{HLR} and is a combination of \cite[Theorem 2.2.5]{HRV} and \cite[Proposition 5.2.8]{EGO}. We only indicate the necessary modifications.

We first consider the branched case. We regard $\Rep_{\Gamma,\mathcal{C}}(X,G)$ as a subvariety of $\GL_n^{2g}\times\prod_jC_j$ and consider the map 
\begin{equation*}
\begin{split}
\mu:\GL_n^{2g}\times\prod_jC_j&\longrightarrow\GL_n\\
(A_i,B_i)_i(Y_j\sigma)_j&\longmapsto\prod_i[A_i,B_i]\prod_j(Y_j\sigma).
\end{split}
\end{equation*}
We will prove that its differential $\ud\mu$ is surjective onto $\mathfrak{gl}_n$ at any point of $\Rep_{\Gamma,\mathcal{C}}(X,G)$. Note that for the usual $\GL_n$-character variety, the image is always contained in $\mathfrak{sl}_n$. A tangent vector at $Y_j\sigma\in C_j$ can be written as $[P_j,Y_j\sigma]$ for some matrix $P_j\in\mathfrak{gl}_n$. In this expression, we understand that $\sigma P_j$ can be identified with $-P_j^t\sigma$ where $P_j^t$ means the transpose. Then for any tangent vector $(a_i,b_i)_i([P_j,Y_j\sigma])_j$ at a point $(A_i,B_i)_i(Y_j\sigma)_j$ of $\Rep_{\Gamma,\mathcal{C}}(X,G)$, we have
\begingroup
\allowdisplaybreaks
\begin{align}
\nonumber
&\ud\mu\left((a_i,b_i)_i([P_j,Y_j\sigma])_j\right)\\
\nonumber
=&\sum_1^g[A_1,B_1]\cdots[A_{i-1},B_{i-1}]a_iB_iA_i^{-1}B_i^{-1}[A_{i+1},B_{i+1}]\cdots[A_g,B_g]\times\prod_j^{2k}(Y_j\sigma)\\
\nonumber
&+\sum_1^g[A_1,B_1]\cdots[A_{i-1},B_{i-1}]A_ib_iA_i^{-1}B_i^{-1}[A_{i+1},B_{i+1}]\cdots[A_g,B_g]\times\prod_j^{2k}(Y_j\sigma)\\
\nonumber
&-\sum_1^g[A_1,B_1]\cdots[A_{i-1},B_{i-1}]A_iB_iA_i^{-1}a_iA_i^{-1}B_i^{-1}[A_{i+1},B_{i+1}]\cdots[A_g,B_g]\times\prod_j^{2k}(Y_j\sigma)\\
\nonumber
&-\sum_1^g[A_1,B_1]\cdots[A_{i-1},B_{i-1}]A_iB_iA_i^{-1}B_i^{-1}b_iB_i^{-1}[A_{i+1},B_{i+1}]\cdots[A_g,B_g]\times\prod_j^{2k}(Y_j\sigma)\\
\nonumber
&+\sum_{j=1}^{2k}\prod_{i=1}^g[A_i,B_i]Y_1\sigma\cdots Y_{j-1}\sigma[P_j,Y_j\sigma]Y_{j+1}\sigma\cdots Y_{2k}\sigma.
\end{align}
\endgroup
Now we take $Z\in\mathfrak{gl}_n$ and assume $\tr(Z\Ima\ud\mu)=0$. The arguments for usual character varieties show that $Z$ must commute with $A_i$, $B_i$, $Y_j\sigma$, $1\le i\le g$, $1\le j\le 2k$. 

Observe that $A_i$, $B_i$, $Y_j\sigma$, $1\le i\le g$, $1\le j\le 2k$, generate the image $H$ of $\pi_1(X)\rightarrow\GL_n\rtimes\lb\sigma\rb$, which is also generated by the image of $\pi_1(\tilde{X})$ and one extra element of $\pi_1(X)\setminus\pi_1(\tilde{X})$. Besides, by Proposition \ref{Gen->Irr}, $H$ is an irreducible subgroup. Now we are in the situation of Proposition \ref{irrsbgpGLsigma}.  Applying Remark \ref{LieirrsbgpGLsigma} to $H$, we deduce that $Z=0$. We conclude that $\ud\mu$ is surjective, and the dimension formula follows.

In the case of unbranched coverings, we define
\begin{equation*}
\begin{split}
\mu:\GL_n^{2g}\times\prod_jC_j&\longrightarrow\GL_n\\
(A_i,B_i)_i(X_j)_j&\longmapsto A_1\sigma(B_1)A_1^{-1}B_1^{-1}\prod_{i=2}^g[A_i,B_i]\prod_jX_j.
\end{split}
\end{equation*}
We will only explain how the vanishing of trace $\tr(Z\Ima\ud\mu)=0$ implies that $Z$ commutes with $A_1\sigma$. Recall that $\sigma(g)=Jg^{-t}J^{-1}$. Taking the derivatives with respect to $A_1$ and $B_1$, we get the following two summands of $\ud\mu$:
\begingroup
\allowdisplaybreaks
\begin{align*}
f(a)&=aA_1^{-1}-A_1\sigma(B_1)A_1^{-1}a\sigma(B_1)^{-1}A_1^{-1},\\
g(b)&=A_1\sigma(B_1)\sigma(b)A_1^{-1}-A_1\sigma(B_1)A_1^{-1}B_1^{-1}bA_1\sigma(B_1)^{-1}A_1^{-1},
\end{align*}
\endgroup
with $a$ and $b$ lying in $\gl_n$, and $\sigma(b):=-Jb^{t}J^{-1}$. From $\tr(Zf(a))=0$ for any $a$, we deduce that $Z$ commutes with $A_1\sigma(B_1)A_1^{-1}$. Then we calculate
\begingroup
\allowdisplaybreaks
\begin{align*}
&\tr(ZA_1\sigma(B_1)\sigma(b)A_1^{-1})\\
=&-\tr(ZA_1JB_1^{-t}b^tJ^{-1}A_1^{-1})\\
=&-\tr(A_1^{-t}J^{-1}bB_1^{-1}JA_1^tZ^t)\\
=&-\tr(B_1^{-1}JA_1^tZ^tA_1^{-t}J^{-1}b),
\end{align*}
\endgroup
and $$\tr(ZA_1\sigma(B_1)A_1^{-1}B_1^{-1}bA_1\sigma(B_1)^{-1}A_1^{-1})=\tr(ZB_1^{-1}b).$$ The vanishing of the trace $\tr(-B_1^{-1}JA_1^tZ^tA_1^{-t}J^{-1}b-ZB_1^{-1}b)$ for any $b$ implies $$B_1^{-1}\sigma(A_1)^{-1}\sigma(Z)\sigma(A_1)B_1=Z.$$ Together with the commutativity with $A_1\sigma(B_1)A_1^{-1}$, this shows that $Z$ commutes with $A_1\sigma$, i.e. $A_1\sigma(Z)A_1^{-1}=Z$. 
\end{proof}

\addtocontents{toc}{\protect\setcounter{tocdepth}{1}}
\bibliographystyle{alpha}
\bibliography{GammaG}
\end{document}